\documentclass[11pt]{article}

\usepackage[]{graphicx,float,latexsym,times}
\usepackage{amsfonts,amstext,amsmath,amssymb,amsthm}
\usepackage{caption2}

\usepackage{macrosBonus}

\newcommand{\IHat}{\hat{\imath}_{\tau_\delta}}

\long\def\symbolfootnote[#1]#2{\begingroup%
\def\thefootnote{\fnsymbol{footnote}}\footnote[#1]{#2}\endgroup}

\usepackage{titlesec} 

\titleformat{\section}{\large\bfseries}{\thesection.}{.5em}{}
\titlespacing*{\section}{0pt}{*3}{*2}
\titleformat{\subsection}{\normalfont\bfseries}{\thesubsection.}{.5em}{}
\titlespacing*{\subsection} {0pt}{*3}{*2}
\titleformat{\subsubsection}{\normalfont\bfseries}{\thesubsubsection.}{.5em}{}
\titlespacing*{\subsubsection} {0pt}{*3}{*2}

\theoremstyle{plain} 
\newtheorem{theorem}{Theorem}[section]
\newtheorem{lemma}{Lemma}[section]

\theoremstyle{definition} 
\newtheorem{definition}{Definition}[section]
\newtheorem{remark}{Remark}[section]

\numberwithin{equation}{section} 

\usepackage[authoryear]{natbib}

\begin{document}

\title{Non-Asymptotic Sequential Tests for Overlapping Hypotheses Applied to Near Optimal Arm Identification in Bandit Models}

\date{}

\maketitle


\author{
\begin{center}
\vskip -1cm

\textbf{\large Aurélien Garivier$^1$ and Emilie Kaufmann$^2$}

$^1$ Unit\'e de Math\'ematiques Pures et Appliqu\'ees et Laboratoire de l'Informatique du Parallélisme\\
				\'Ecole Normale Sup\'erieure de Lyon, Universit\'e de Lyon, Lyon, France \\
$^2$ CNRS \& Univ. Lille, Inria SequeL, UMR 9189- CRIStAL, Lille, France 
\end{center}
}

\symbolfootnote[0]{\normalsize Address correspondence to E. Kaufmann,
Inria Lille Nord-Europe, 40 avenue Halley, 59650 Villeneuve d'Ascq, France; E-mail: emilie.kaufmann@univ-lille.fr}

{\small \noindent\textbf{Abstract:} In this paper, we study sequential testing problems with \emph{overlapping} hypotheses. 
			We first focus on the simple problem of assessing if the mean $\mu$ of a Gaussian distribution is smaller or larger than a fixed $\epsilon>0$; if $\mu\in(-\epsilon,\epsilon)$, both answers are considered to be correct. 
			Then, we consider PAC-best arm identification in a bandit model: given $K$ probability distributions on $\R$ with means $\mu_1,\dots,\mu_K$, we derive the asymptotic complexity of identifying, with risk at most $\delta$, an index $I\in\{1,\dots,K\}$ such that $\mu_I\geq \max_i\mu_i -\epsilon$. We provide non asymptotic bounds on the error of a parallel General Likelihood Ratio Test, which can also be used for more general testing problems. We further propose lower bound on the number of observation needed to identify a correct hypothesis. 
			Those lower bounds rely on information-theoretic arguments, and specifically on two versions of a change of measure lemma (a high-level form, and a low-level form) whose relative merits are discussed.}

\section{INTRODUCTION} \label{s:Intro}

	In order to extract information on a random source $\bP$, the classical statistical framework relies on a fixed sample $X_1,\dots,X_n$ of a given size $n$. For example, a classical hypothesis test is a (possibly randomized) function $\hat{\imath}_n$ of the sample. In \emph{sequential statistics}, on the contrary, the statistician is allowed to adaptively decide on the sample size. For example, the sequential hypothesis testing of $\cH_0$ versus $\cH_1$ consists of a stopping time $\tau$ with respect to the filtration $(\mathcal{F}_t)_t$, where $\mathcal{F}_t=\sigma(X_1,\dots,X_t)$, and a decision rule $\hat{\imath}_\tau$ which is $\mathcal{F}_\tau$-measurable. 
	
	Even more power is devoted to the statistician in \emph{active testing}: the statistician not only chooses when to stop, but also sequentially influences the observations. For example, in the \emph{multi-armed bandit problem}, a set $\mathcal{A}$ of unknown random sources is available on demand, and the statistician chooses one of them $A_t\in\mathcal{A}$ at every time step $t$, from which she gets a new sample. The choice of $A_t$ is made based only upon past observations: $A_t$ is $\mathcal{F}_{t-1}$-measurable. The goal of the statistician is typically \emph{best-arm identification}, that is to infer which source has the highest expectation using as few samples as possible. Given a risk parameter $\delta$, a strategy (consisting of a sampling, a stopping and a decision rule) is called \emph{$\delta$-correct} if the probability that the decision $\hat{a}_\tau$ is not the source with highest expectation is smaller than $\delta$. The \emph{sample complexity} is the smallest achievable expectation for the stopping time of a $\delta$-correct strategy. 
	
	In the case of a finite set of sources in a one-dimensional canonical exponential model, the complexity of the best-arm identification problem is well understood since \citep{GK16,Russo16}. A non-asymptotic lower bound has been given on the sample complexity: the expectation of the stopping of any $\delta$-correct strategy is always larger than the product of a model-dependent \emph{characteristic time}, which can be computed as the solution of a convex optimization problem depending on the sources, and a confidence term which scales like $\ln(1/\delta)$. 
	This lower-bound is tight: a strategy asymptotically matching it when the risk $\delta$ goes to $0$ has been described.
	
	But the problem of $\delta$-correct best-arm identification is not really well-posed: the characteristic time of a model tends to infinity when one of the sources has an expectation almost as high as the best one, and $\delta$-correct strategies never stop (with high probability) if two sources are simultaneously optimal. In fact, a much more natural and useful objective is to search for a \emph{Probably Approximately Correct} (PAC) decision: given a risk parameter $\delta$ and a tolerance parameter $\epsilon>0$, a strategy is called $(\epsilon,\delta)$-PAC if for every bandit model of the considered class the probability that the chosen source $\hat{a}_\tau$ has a mean smaller that the highest one \emph{minus $\epsilon$} is smaller than $\delta$. Note that a strategy is $\delta$-correct if it is $(0,\delta)$-PAC: the latter notion is thus a generalization of the former.
	
	The purpose of this paper is to show to what extend the results on the complexity of $\delta$-correct  best-arm identification can be generalized to the PAC setting.
	An interesting feature of the latter is that several decisions may simultaneously be admissible if several source means are within $\epsilon$ of their maximum. In other words, this supposes to deal with \emph{sequential testing of overlapping hypotheses}, a surprisingly little investigated problem of interest on its own. 
	
	The paper is organized as follows: Section~\ref{sec:overhyp} is devoted to the general problem of sequential testing of overlapping hypotheses. A generic sequential test based on generalized likelihood ratios (called the parallel GLRT) is introduced. Section~\ref{sec:twogaussians} details the simplest example: testing if the mean of a Gaussian distribution is smaller than $\epsilon$ or larger that $-\epsilon$, for which we propose a $\delta$-correct parallel GLRT with optimal sample complexity. We show in Section~\ref{sec:PACbandit} that parallel GLRTs remain $\delta$-correct for very general active identification problems in multi-armed bandit models. The rest of the paper investigates the sample complexity of an important active identification problem: PAC best-arm identification. Section~\ref{sec:lbPACBAI} provides our main result on the complexity of PAC best-arm identification. In a nutshell, it appears that a complexity generalizing the results known for $\delta$-correct procedures can be identified, but only asymptotically, under stronger hypothesis, and using less elegant arguments: high-level information-theoretic reasoning cannot be applied here. In Section~\ref{sec:OptimalStrategy} we then propose an $(\epsilon,\delta)$-PAC strategy that combines the parallel GLRT with a ``tracking'' sampling rule designed for converging to an optimal allocation. This strategy is matching our lower bound for a large family of regular bandit instances.

\section{THE PARALLEL GLRT FOR TESTING OVERLAPPING HYPOTHESES}\label{sec:overhyp}
	
	\subsection{From Sequential to Active Testing}
	
	In a classical parametric hypothesis testing framework, some data $D$ is collected whose distribution depends on some parameter $\mu \in \cR$ and is denoted by $\bP_{\mu}$. Given two hypotheses 
	\[\cH_0 : (\mu \in \cR_0) \ \ \ \text{and} \ \ \ \cH_1 : (\mu \in \cR_1)\;,\]
	where $\cR_0$ and $\cR_1$ are two \emph{disjoint} subsets of $\cR$, the goal is to build a \emph{test} $\Phi$  mapping the data to a decision in $\{0,1\}$, where $\Phi(D)=1$ means that we \emph{reject} hypothesis $\cH_0$ based on the observed data. 
	Usually, the data collected consist of a fixed number $n$ of i.i.d. samples from some distribution depending on $\mu$, and the two hypotheses do not play symmetric roles: the risk to reject the null hypothesis $H_0$ when it is true, called the type I error $\alpha := \sup_{\mu \in \cR_0} \bP_{\mu} (\Phi(D) = 1)\;$,
	is controlled, and the decision $\Phi(D)=0$ is not interpreted as an acceptation of $\cH_0$, but as a absence of rejection. Under a controlled type I error, the procedure aims at minimizing the type II error $\beta := \sup_{\mu \in \cR_1} \bP_{\mu} (\Phi(D) = 0)$,
	that is at maximizing \emph{power} $1 - \beta$ of the test.

	In some cases, however, the hypotheses do play symmetric roles and both type I and type II errors have the same importance. When this happens, we really want to \emph{decide for} $\cH_0$ or $\cH_1$. This distinction is of importance in what follows, and both cases will be illustrated in our approach.

	We focus in this paper on \emph{sequential testing} for \emph{symmetric hypotheses}. 
	In order to emphasize that there is no preferred (null) hypothesis, the hypotheses are denoted by $\cH_1$ and $\cH_2$. Given a data stream $X_1,X_2,\dots$ and two hypotheses $\cH_1: (\mu \in \cR_1)$ and $\cH_2: (\mu \in \cR_2)$, the goal is to stop the data collection process after some (random) number $\tau$ of observations, and based on this observation $D_\tau = (X_1,\dots,X_\tau)$ make a decision $\hat{\imath}_\tau = \Phi(D_\tau) \in \{1,2\}$, where $\hat{\imath}_\tau = i$ means that hypothesis $\cH_i$ is selected. 
	The duration $\tau$ of the test should be minimal, but needs to be long enough so that it is actually possible to discriminate between the two hypotheses. This area was pioneered by \cite{Wald45SPRT}, who proposes the Sequential Probability Ratio Test (SPRT) for two simple hypotheses, e.g. in a parametric setting $\cH_1 : (\mu = \mu_1)$ and $\cH_2 : (\mu = \mu_2)$. Among the sequential tests with prescribed type I and type II errors, the SPRT is proved to have the smallest average duration $\bE_{\mu_i}[\tau]$ for $i \in \{1,2\}$. Later, particular examples of sequential test of \emph{composite} hypotheses (in which $\cR_1$ and $\cR_2$ are not reduced to a singleton) have also been studied, see e.g.~\citep{RobbinsSiegmund74,Lai88}.  
	
	We introduce below in Section~\ref{subsec:Framework} a broader sequential testing framework, in which we allow for composite hypotheses that are \emph{possibly overlapping}, allowing for $\cR_1 \cap \cR_2 \neq \emptyset$. Moreover, we allow to simultaneously test more than two hypotheses. In this setup, the goal is to stop the data collection as quickly as possible in order to identify \emph{one} correct hypothesis. We present a generic sequential test for this problem, which is an extension or the classical Generalized Likelihood Ratio test. Then, we illustrate its efficiency  in the context of \emph{active testing}, especially for active identification problems in a multi-armed bandit model.  
	
	Active testing can be traced back to the early work of \cite{Chernoff59}: at each round, \emph{an experiment is selected and performed} in order to gain information about which of a finite set of hypotheses is true. An active test is thus made of a \emph{sampling rule} (that indicates which experiment $A_t$ is selected at round $t$) together with the stopping rule and recommendation rules $(\tau,\hat\imath)$ inherent to sequential testing. In a multi-armed bandit model, these experiments consists in sampling one of the $K$ available \emph{arms} (probability distributions) $\nu_1,\dots,\nu_K$ with respective means $\mu_1,\dots,\mu_K$. As explained in Section~\ref{sec:PACbandit} below, we assume here that the arms belong to a one-parameter exponential family, and as such can be parameterized by their means. 
	
	Many works on bandit models in the past decade have focus on the reinforcement learning framework, where the samples collected are viewed as rewards and the goal is to maximize cumulative rewards (see, e.g. \cite{Robbins52Freq,LaiRobbins85bandits,Bubeck:Survey12}). A broader class of objectives is to \emph{learn something} about the unknown parameter $\bm \mu = (\mu_1,\dots,\mu_K)$ and several active testing problems have been studied, including the $\epsilon$-best arm identification problem \citep{EvenDaral06} where the goal is to find any arm whose expectation is within $\epsilon$ of the highest mean. We show in Section~\ref{sec:PACbandit} that the parallel GLRT test provides a valid stopping rule for very general tests in bandit models. We investigate in Section~\ref{sec:OptimalStrategy} how it can help reaching the optimal sample complexity for $\epsilon$-best arm identification, when coupled with a good sampling rule.

	\subsection{Sequential Test of Overlapping Hypotheses}\label{subsec:Framework}

	We formally define our framework in a rather general setting in which one sequentially collects samples $X_1,X_2,\dots$ and the joint distribution of these samples depends on some parameter $\mu \in \cR$. Let $\cR_1, \cR_2, \dots, \cR_M$ be $M$ regions that cover the parameter space (i.e. such that $\bigcup_{i=1}^M\cR_i = \cR$) but that do not necessarily form a partition of $\cR$. We consider the corresponding $M$ hypotheses \[\cH_1 : (\mu \in \cR_1) \ \ \ \ \ \cH_2 : (\mu \in \cR_2) \ \ \ \dots \ \ \ \cH_M : (\mu \in \cR_M)\;.\] 
	We denote by $(\cF_t =\sigma(X_1,\dots,X_t))_{t\geq 1}$ the filtration generated by the observations.
	A sequential test of the hypotheses $\cH_1,\dots,\cH_M$ based on the stream of samples $(X_s)_{s \in \N}$ is a pair $(\tau,\hat{\imath}_\tau)$ where $\tau$ is a \emph{stopping rule} indicating when to stop the test (hence a is a stopping time with respect to $(\cF_t)_t$), and $\hat{\imath}_\tau \in \{1,\dots,M\}$ is a \emph{recommendation rule} that proposes upon stopping a guess for a correct hypothesis  based on the observed data (hence $\cF_{\tau}$-measurable).
	
	\begin{definition} Let $\delta \in (0,1)$ be a risk parameter. A sequential test $(\tau_{\delta}, \hat{\imath}_{\tau_\delta})$ is \emph{ $\delta$-correct } if for all  $\mu \in \cR$, \[\bP_\mu\left(\tau_\delta < \infty, \mu \notin \cR_{\hat{\imath}_{\tau_\delta}}\right)\leq \delta\;.\]
	\end{definition}
	
	In words, the probability that the test stops and recommends a wrong hypothesis $\cH_i$ (such that $\mu \notin \cR_i$) should be under control. For example, in the presence of two overlapping hypotheses, the test should be such that  $\bP_{\mu}(\IHat = 2) \leq \delta$ for all $\mu \in \cR_1\backslash \cR_2$ and $\bP_{\mu}(\IHat = 1) \leq \delta$ for all  $\mu \in \cR_2\backslash \cR_1$; however, for all $\mu \in \cR_1 \cap \cR_2$ the test may decide for either hypothesis (at any time). 
	Among all $\delta$-correct sequential tests, the preferred one is the one that stops after using as few samples as possible: we want $\tau_\delta$ to be small in expectation for every $\mu \in \cR$.

	\subsection{The Parallel Generalized Likelihood Ratio Test}
	
	The likelihood of the observations collected up to round $t$ under the distribution parameter $\lambda\in\cR$ is denoted by $\ell(X_1,\dots, X_t ; \lambda)$. 
	To decide between the $M$ symmetric hypotheses, we run in parallel $M$  sequential tests of the following two non-overlapping (and asymmetric) hypotheses
	\[\tilde{\cH}_0 : ( \mu \in \cR\backslash \cR_i) \ \ \ \text{against} \ \ \ \tilde{\cH}_1: ( \mu \in \cR_i),\]
	for each $i \in \{1,\dots,M\}$. We stop when one of these tests rejects hypothesis $\tilde{\cH}_0$. The $i$-th test causing to stop means that $\mu$ is believed to belong to $\cR_i$ (this does not exclude that it may also belong to other regions), in which case we set $\hat{\imath}_{\tau} = i$ . 
	
	Our test of $\tilde{\cH}_0$ against $\tilde{\cH}_1$ is a Generalized Likelihood Ratio Test (GLRT). The Generalized Likelihood Ratio statistic based on $t$ samples is defined as 
	\[\frac{\max_{\lambda \in \cR}\ell(X_1,\dots, X_t ; \lambda)}{\max_{\lambda \in \cR \backslash \cR_i}\ell(X_1,\dots, X_t ; \lambda)} = \inf_{\lambda \in \cR\backslash \cR_i} \frac{\ell(X_1,\dots, X_t ; \hat{\mu}(t))}{\ell(X_1,\dots, X_t ; \lambda)}\;,\]
	where $\hat{\mu}(t)$ is the maximum likelihood estimator (in $\cR$). Large values of this GLR statistic tend to reject $\tilde{\cH}_0$. Calibrating the rejection threshold for a GLR based on a fixed sample size is often done by resorting to asymptotic arguments (like the Wilks phenomenon \citep{Wilks38}) describing the limit distribution of the GLR under the null hypothesis; this is however not useful for the finite-confidence bandit analysis that follows. We propose to use a threshold function $\beta(t,\delta)$ that depends on the current number of samples $t$ and on the risk parameter $\delta$. We provide in the cases considered here a valid choice for the threshold $\beta(t,\delta)$ in order to ensure $\delta$-correctness of the corresponding test.  
	
	The parallel GLRT using the threshold function $\beta(t,\delta)$ is formally defined in the following way. Given $\delta \in (0,1)$, the stopping rule is 
	\begin{equation}\tau_\delta = \inf \left\{ t \in \N : \max_{i = 1,\dots,M} \inf_{\lambda \in \cR \backslash \cR_i} \ln \frac{\ell(X_1,\dots,X_t ; \hat{\mu}(t))}{\ell(X_1,\dots,X_t ; \lambda)} > \beta(t,\delta)\right\}\label{def:Stop}\end{equation}
	and the decision rule is
	\begin{equation}\IHat \in \argmax{i = 1,\dots,M} \inf_{\lambda \in \cR \backslash \cR_i} \ln \frac{\ell(X_1,\dots,X_\tau ; \hat{\mu}(\tau_\delta))}{\ell(X_1,\dots,X_t ; \lambda)}\label{def:Decision}\end{equation}
	(ties can be resolved arbitrarily). Note that the maximum over $i \in \{1,\dots,M\}$ in the definition of $\tau_\delta$ and $\IHat$ can be reduced to the set of hypotheses to which $\hat{\mu}(t)$ belongs.
	
	In the remainder of the paper we propose different choices of the threshold function that guarantee $\delta$-correctness for the corresponding parallel GLRT. We first study a simple example in Section~\ref{sec:twogaussians}: based on i.i.d. Gaussian samples, the goal is to decide whether the mean is ``rather positive'' or ``rather negative''. We then generalize it in Section~\ref{sec:PACbandit} where we propose a parallel GLRT which is $\delta$-correct for any active identification problem in a multi-armed bandit model. 	
	In Section~\ref{sec:lbPACBAI} and \ref{sec:OptimalStrategy}, we focus on a particular active identification problem of practical interest: $\epsilon$-best arm identification. In this context, we further show that a parallel GLRT is a natural candidate to minimize the sample complexity, when coupled with an appropriate sampling rule.

\section{SEQUENTIAL TEST OF OVERLAPPING HYPOTHESES: AN ARCHETYPICAL EXAMPLE}\label{sec:twogaussians}

	We focus here in this section on a stream of independent samples $X_1,X_2,\dots$ of the distribution $\norm{\mu}{\sigma^2}$, where the variance $\sigma^2$ is known. The statistical problem is to determine whether $\mu$ is positive or negative, but with and indifference region of width $2\epsilon$: we consider the two hypotheses
	\[\cH_1 : (\mu < \epsilon) \ \ \text{ and } \  \ \cH_2 : (\mu > -\epsilon)\;.\]
	The hypotheses $\cH_1$ and $\cH_2$ are not mutually exclusive. In this particular example, one aims at building a stopping rule $\tau_\delta$ and a recommendation rule $\IHat$ such that 
	\begin{equation}\forall \mu \leq - \epsilon, \ \bP_\mu(\IHat = 2) \leq \delta \quad \ \text{and} \quad \ \ \forall \mu \geq \epsilon, \ \bP_\mu(\IHat = 1) \leq \delta\;,\label{requirements}\end{equation}
	but any answer $\IHat$ is considered correct when $\mu\in\cR_1 \cap \cR_2 =(-\epsilon,\epsilon)$.

	Such a test can be useful when we simultaneously collect samples from two Gaussian populations with means $\mu_1$ and $\mu_2$ in order to decide which of the two has the largest mean, up to an accuracy of $\epsilon$ (that is whether, $\mu_1 > \mu_2 - \epsilon$ or $\mu_2 > \mu_1 - \epsilon$). This corresponds to solving the $\epsilon$-best arm identification problem in a two-armed Gaussian bandit model (see Section~\ref{sec:PACbandit}) based on uniformly sampling the two arms (which turns out to be optimal in this very particular case). The reason we start with this simple example is twofolds: it provides a first concrete example of parallel GLRT that is easy to analyse, and it permit to showcase our tools to study the optimal sample complexity of a test.

	\subsection{Properties of the Parallel GLRT Test} \label{subsec:PropGLRTGaussian}
	
	The expression of the parallel GLRT is easy to compute in this Gaussian case. Observing that
	\[\ln \frac{\ell(X_1,\dots, X_t ; \hat{\mu}(t))}{\ell(X_1,\dots, X_t ; \lambda)} = \frac{t(\hat{\mu}_t - \lambda)^2}{2\sigma^2}\]
	where $\hat{\mu}_s = \frac{1}{s}\sum_{i=1}^s X_i$ denote the empirical mean of the observations, and that $\cR\backslash \cR_1 = [\epsilon, +\infty)$ and $\cR\backslash \cR_2 = (-\infty, - \epsilon]$, the expressions given in \eqref{def:Stop} and \eqref{def:Decision} particularize to
	\begin{eqnarray}\tau_\delta &= &\inf \left\{t \in \N : \frac{t(|\hat{\mu}_t| + \epsilon)^2}{2\sigma^2} > \beta(t,\delta) \right\}\;,\label{strategy} \\
	\IHat &= & 2 \ \ \text{if and only if} \ (\hat{\mu}_{\tau_\delta} >0).
	\nonumber
	\end{eqnarray}
	
	\paragraph{Correctness} The $\delta$-correctness of the parallel GLRT follows from a more general argument developed in Section~\ref{sec:PACbandit} for more general active identification problems in bandit models (see Remark \ref{rem:Gaussian} therein). Lemma~\ref{lem:deltaCorrect} indeed permits to propose a threshold for which the above test is $\delta$-correct:
	\begin{equation}\beta(t,\delta) = 3\ln(\ln(t)+1) + \cT\left(\ln(1/\delta)\right),\label{ThresholdGLRTGaussian}\end{equation}
	where the function $\cT$ is such that $\cT(x) \simeq x + \ln(x)$. 
	
	\paragraph{Sample Complexity} Under the simple data collection process in this particular case, it is possible to give an upper bound on the expected number of samples used by the parallel GLRT that employs the threshold \ref{ThresholdGLRTGaussian} before stopping and making a decision.
	
	We first give a crude, asymptotic analysis. Fix $\mu \in \R$ and let $\alpha \in [0,\epsilon)$. 
	\begin{eqnarray*}
		\bE[\tau_\delta] & \leq & \sum_{t=1}^\infty \bP\left(t(|\hat{\mu}_t| + \epsilon)^2 \leq 2\sigma^2\beta(t,\delta)\right) \\
		& \leq & \sum_{t=1}^\infty \bP\left( |\hat{\mu}_t - \mu | > \alpha\right)+ \sum_{t=1}^\infty \bP\left(t(|\hat{\mu}_t| + \epsilon)^2 \leq 2\sigma^2\beta(t,\delta), |\hat{\mu}_t - \mu | \leq \alpha\right) \\
		& \leq & \sum_{t=1}^\infty \bP\left( |\hat{\mu}_t - \mu | > \alpha\right)+ \sum_{t=1}^\infty \bP\left(t(|\mu| - \alpha + \epsilon)^2 \leq 2\sigma^2\beta(t,\delta), |\hat{\mu}_t - \mu | \leq \alpha\right)\,.
	\end{eqnarray*}
	The first term is upper bounded by a constant (independent of $\delta$), while the second is upper bounded by 
	\[T_0(\delta) = \inf\left\{ T \in \N^* : \forall t \geq T , t(|\mu| - \alpha + \epsilon)^2 \leq 2\sigma^2\beta(t,\delta)\right\}\:.\]
	For $\beta(t,\delta)$ as given in \eqref{ThresholdGLRTGaussian}, Lemma \ref{lem:Aurelien} in Appendix~\ref{proof:GaussianNonAsymptotic}  shows that 
	\[T_0(\delta) = \frac{2\sigma^2}{(|\mu| - \alpha + \epsilon)^2}\ln \frac{1}{\delta} + o_{\delta \rightarrow 0}\left(\ln \frac{1}{\delta}\right)\;.\]
	Letting $\alpha$ go to zero, one obtains, for all $\mu \in \R$,
	\[\limsup_{\delta \rightarrow 0} \frac{\bE_\mu[\tau_\delta]}{\ln(1/\delta)} \leq \frac{2\sigma^2}{(|\mu|+\epsilon)^2}\;.\]
	
	The following more precise, non-asymptotic statement is proved in Appendix \ref{proof:GaussianNonAsymptotic}.

	\begin{theorem}\label{thm:GaussianNonAsymptotic} Introducing the notation $\ell(\delta) = \cT(\ln(1/\delta)) + \frac{3}{e}$, the parallel GLRT with threshold function $\beta(t,\delta)$ given in \eqref{ThresholdGLRTGaussian} satisfies
		\begin{align*}&\bE[\tau_\delta]  \leq  \frac{2\sigma^2}{(|\mu|+\epsilon)^2}\left[\ell(\delta) + \frac{6}{e}\ln\left(\frac{2\sigma^2}{(|\mu|+\epsilon)^2}\ell(\delta)\right)  + 8\sqrt{\ell(\delta) + \frac{6}{e} \ln\left(\frac{2\sigma^2}{(|\mu|+\epsilon)^2}\ell(\delta)\right)}+ {32\delta^{1/8}}\right]+ 1\;.
		\end{align*}
	\end{theorem}

	To summarize, the parallel GLRT with threshold function \eqref{ThresholdGLRTGaussian} is $\delta$-correct and the expected number of samples uses satisfies $\bE_{\mu}\left[\tau_\delta\right] \lesssim \frac{2\sigma^2}{(|\mu|+\epsilon)^2}\ln\left(\frac{1}{\delta}\right)$ for small values of the risk parameter $\delta$, for every $\mu \in \R$, even in the indifference zone $(-\epsilon,\epsilon)$. We prove in the next section that this is the smallest possible number of samples for a $\delta$-correct test.
	
	\subsection{Asymptotic Optimality of the Parallel GLRT}
	
	In order to prove the optimality of the parallel GLRT, we need to provide a \emph{lower bound} on the sample complexity of any $\delta$-correct algorithm. 
	Let us emphasize that we are not interested in minimax lower bounds (see e.g. \cite{tsyb08np}), but in problem-dependent lower bounds providing, for any (reasonable) algorithm, the minimal number of samples necessary to take a decision with the right precision and risk \emph{for each possible value of the parameters}. Such a result is typically obtained by a \emph{change of distribution} argument. 

	A change of distribution consists in relating the probability of an event $C$ under two different probability distributions. 
	It is well-known that for any event $C\in\cF_t$, 
	
	\begin{align}
\bP_{\lambda}(C) &=  \bE_{\lambda}[\ind_{C}] = \int \ind_C(x_1,\dots,x_t) \,d\bP_{\lambda}^{X_1,\dots,X_n}(x_1,\dots,x_n)\nonumber \\
	&=  \int \ind_C(x_1,\dots,x_t) \frac{d\bP_{\lambda}^{X_1,\dots,X_t}(x_1,\dots,x_t)}{d\bP_{\mu}^{X_1,\dots,X_t}(x_1,\dots,x_t)} \,d\bP_{\mu}^{X_1,\dots,X_t}(x_1,\dots,x_t) \nonumber\\
	&=  \int \ind_C(x_1,\dots,x_t) \frac{\ell(X_1,\dots,X_t ; \lambda)}{\ell(X_1,\dots,X_t ; \mu)} \,d\bP_{\mu}^{X_1,\dots,X_t}(x_1,\dots,x_t)\nonumber \\
	& =  \bE_{\mu}\Big[\ind_{C} \exp\big(-L_t(\mu,\lambda)\big)\Big]\nonumber,
	\end{align}
	where $L_t(\mu,\lambda)$ denotes the log-likelihood ratio of the observations:
	\[L_t(\mu,\lambda) := \ln \frac{\ell(X_1,\dots,X_t ; \mu)}{\ell(X_1,\dots,X_t ; \lambda)}.\]
	This elementary change of distribution has been a key ingredient for the derivation of lower bounds in the bandit literature. The most notable example is the famous regret lower bound by Lai and Robbins in the seminal paper \citep{LaiRobbins85bandits}, and further examples include notably \citep{Bubeck10BestArm}.
	
	But this ``low-level" ingredient can also serve as a basis for more elaborate arguments, which lead to more elegant and stronger results as for example illustrated in \citep{GK16}.
	We propose in Lemma~\ref{lem:CD} below two forms of change of distribution that are useful: the low-level form~\eqref{WeakConverse}, and the high-level form~\eqref{StrongConverse}. 
	The high-level form, of information-theoretic flavor, involves the Kullback-Leibler divergence  $\mathrm{KL}(\cdot,\cdot)$ from one distribution to another. For any random variable $Z$ defined on a probabilistic space $(\Omega,\mathcal{F}, \bP)$, we denote by  $\bP^Z$ the law of $Z$.

	\begin{lemma} \label{lem:CD}
		Let $\mu$ and $\lambda$ be two parameters. 
		\begin{enumerate}
			\item \textbf{High-level form:} for any stopping time $\tau$ and any event $C \in \cF_\tau$,
			\begin{equation}\bE_{\mu}[L_{\tau}(\mu,\lambda)] =\mathrm{KL}\left(\bP_\mu^{X_1,\dots,X_\tau}, \bP_\lambda^{X_1,\dots,X_\tau}\right) \geq \mathrm{kl}\big(\bP_{\mu}(C) , \bP_{\lambda}(C)\big)\;,\label{WeakConverse}\end{equation}
			where $\mathrm{kl}(x,y) = x \ln(x/y) + (1-x)\ln((1-x)/(1-y))$ is the binary relative entropy.
			\item \textbf{Low-level form:} for all $x \in \R$, $n \in \N^*$, for all event $C \in \cF_n$,  \begin{equation}\bP_{\lambda}(C) \geq e^{-x}\;\big[\bP_{\mu}(C) - \bP_{\mu}\left(L_n(\mu,\lambda) \geq x \right)\big]\;.\label{StrongConverse}\end{equation}
		\end{enumerate}
		
	\end{lemma}
	
	\begin{proof} 
		The data-processing inequality states that, if $P$ and $Q$ are some probability distributions on the same measurable space $\cX$ and if $f:\cX\to\cY$ is a measurable function, then $\mathrm{KL}(P, Q)\geq\mathrm{KL}(P^f, Q^f)$, where $P^f$ (resp $Q^f$) denotes the push-forward measure of $P$ (resp. $Q$) by $f$. It is a simple consequence of Jensen's inequality applied to the function $x\mapsto x\log(x)$.
		The high-level form is a consequence of the data-processing inequality:
		\begin{equation*}\bE_{\mu}[L_{\tau}(\mu,\lambda)] = \mathrm{KL}\left(\bP_\mu^{X_1,\dots,X_\tau}, \bP_\lambda^{X_1,\dots,X_\tau}\right)
		\geq \mathrm{KL}\left(\bP_\mu^{\ind_C}, \bP_\lambda^{\ind_C}\right)= \mathrm{kl}\big(\bP_{\mu}(C) , \bP_{\lambda}(C)\big)\;,\end{equation*}
		see~\citep{GMS16} for details. The low-level form is even more elementary:
		\begin{eqnarray*}
			\bP_\lambda(C) & = & \bE_{\mu}\Big[\ind_{C} \exp\big(-L_t(\mu,\lambda)\big)\Big] \geq  \bE_{\mu}\Big[\ind_{C}\ind_{(L_t(\mu,\lambda) < x)} \exp\big(-L_t(\mu,\lambda)\big)\Big]\\
			& \geq & e^{-x}\,\bP_{\mu}\Big(C \cap \big(L_t(\mu,\lambda) < x\big)\Big) \\
			& \geq & e^{-x}\,\big[\bP_{\mu}(C) - \bP_{\mu}\left(L_n(\mu,\lambda) \geq x \right)\big]\;.
		\end{eqnarray*}
	\end{proof}
	
	The sample complexity lower bounds provided below illustrate the strengths and weaknesses of each form, that lead to the two statements in Theorem~\ref{thm:Gaussian} below. This result permits to identify the sample complexity of the testing problem with overlapping hypotheses studied in this section. Indeed, Theorem~\ref{thm:Gaussian} shows that any $\delta$-correct sequential procedure satisfies $\bE_{\mu}[\tau_\delta] \geq \frac{2\sigma^2}{(|\mu|+\epsilon)^2} \ln\left(\frac{1}{\delta}\right)$ for every $\mu$, in a regime of small values of $\delta$. In Section~\ref{subsec:PropGLRTGaussian} we exhibited a $\delta$-correct parallel GLRT that uses not more than this number of samples for small values of $\delta$. This proves that the sample complexity of the testing problem, that is the minimal number of samples needed by a $\delta$-correct test, is close to $\frac{2\sigma^2}{(|\mu|+\epsilon)^2} \ln\left(\frac{1}{\delta}\right)$  for small values of the risk $\delta$, and justifies that the parallel GLRT, which attains this sample complexity, is asymptotically optimal. 
	
	It is notable that the lower bound of Theorem~\ref{thm:Gaussian} is non-asymptotic for all $\mu$ that belong to a single hypothesis, which is obtained using the high-level form in Lemma~\ref{lem:CD}, but only asymptotic when $\mu$ is in the intersection $(-\epsilon,\epsilon)$, which requires the low-level form, as can be seen in the proof.
	
	\begin{theorem}\label{thm:Gaussian} Any $\delta$-correct sequential test satisfies 
		\begin{align*}
		\bullet\; &\forall \mu \notin (-\epsilon,\epsilon), \ \ \ \ \ \bE_{\mu}[\tau_\delta] \ \geq \frac{2\sigma^2}{\big(|\mu|+\epsilon\big)^2}\;\mathrm{kl}(\delta,1-\delta)\;,\\ 
		\bullet\;& \forall \mu \in \R, \ \ \liminf_{\delta \rightarrow 0} \frac{\bE_\mu[\tau_\delta]}{\ln(1/\delta)}  \geq \frac{2\sigma^2}{\big(|\mu|+\epsilon\big)^2}\;.
		\end{align*}
	\end{theorem}

	\begin{proof}
		We first treat the case $\mu \notin (-\epsilon,\epsilon)$. In the particular case considered in this section the samples $X_1,\dots,X_t$ are i.i.d. from a Gaussian distribution. Assuming that $\bE_{\mu}[\tau] < \infty$ and using Wald's lemma, the expected log-likelihood ratio takes the simple form
		\[\bE_{\mu}[L_{\tau}(\mu,\lambda)]  = \bE_{\mu}[\tau] \mathrm{KL}\big(\mathcal{N}(\mu,\sigma^2), \mathcal{N}(\lambda, \sigma^2)\big) = \bE_{\mu}[\tau] \frac{(\mu-\lambda)^2}{2\sigma^2}\;.\]
		Hence, the high-level change of measure inequality~\eqref{WeakConverse} translates into the following inequality, which obviously also holds when $\bE_{\mu}[\tau]=\infty$:
		\begin{equation}\bE_{\mu}[\tau] \frac{(\mu - \lambda)^2}{2\sigma^2} \geq \mathrm{kl}\big(\bP_{\mu}(C) \;, \bP_{\lambda}(C)\big) \;.\label{StrongCParticular}\end{equation}
		For $\mu < -\epsilon$,  choosing $\lambda=\epsilon$ and  $C = (\IHat = 2)$, which is such that $\bP_{\mu}(C) \leq \delta$ and $\bP_{\lambda}(C) \geq 1-\delta$, directly yields  
		\begin{equation*}\bE_{\mu}[\tau] \frac{\big(|\mu| + \epsilon\big)^2}{2\sigma^2} \geq \mathrm{kl}(\delta,1-\delta) \;.\label{StrongCParticular2}\end{equation*}
		Similarly, for $\mu > \epsilon$, we use $\lambda=-\epsilon$ and $C = (\IHat = 1)$ so as to obtain the same inequality.
		
		We now turn our attention to the (more interesting) case $\mu \in (-\epsilon,\epsilon)$. For every  $\beta >0$, we prove that  
		\begin{equation}\lim_{\delta \rightarrow 0} \bP_{\mu}\left( \tau_{\delta} \leq \frac{2\sigma^2(1-\beta)}{(|\mu| + \epsilon)^2}\ln \left(1/\delta\right)\right) = 0\;,\label{LB:Proba}\end{equation}
		which entails that 
		\[\liminf_{\delta \rightarrow 0} \frac{\bE_{\bm \mu}[\tau_\delta]}{\ln(1/\delta)} \geq \frac{2\sigma^2(1-\beta)}{(|\mu| + \epsilon)^2}\;.\]
		Letting $\beta$ tend to zero gives the second statement of Theorem~\ref{thm:Gaussian} for $\mu \in (-\epsilon,\epsilon)$. Note that this statement is also true for $\mu \notin (-\epsilon,\epsilon)$, as a consequence of the non-asymptotic lower bound obtained above and the fact that $\mathrm{kl}(\delta,1-\delta) \sim \ln(1/\delta)$ when $\delta$ goes to zero.
		
		We now prove \eqref{LB:Proba} for a fixed $\beta > 0$. Introducing 
		\[n_\delta := \left\lfloor\frac{2\sigma^2(1-\beta)}{\big(|\mu| + \epsilon\big)^2}\ln \frac{1}{\delta}\right\rfloor\;,\]
		and the event $C_\delta = (\tau_\delta \leq n_\delta)$, we need to prove that $ \bP_{\mu}(C_\delta) \to 0$ when $\delta \rightarrow 0$. To do so, we write 
		\[\bP_{\bm\mu}\left(C_\delta\right) = \bP_{\bm \mu}\left(C_\delta, \IHat = 1\right) + \bP_{\bm \mu}\left(C_\delta, \IHat = 2\right)\;,\]
		and prove that the two terms in the RHS tend to zero when $\delta$ goes to zero, using the low-level form of Lemma~\ref{lem:CD} with two different changes of distribution.

		\paragraph{$\bm{\lim_{\delta \rightarrow 0} \bP_{ \mu}\left(C_\delta, \IHat = 1\right) = 0}$}
		Choosing $\lambda = \epsilon$ yields that $\bP_{\lambda}(C_\delta,\IHat = 1) \leq \bP_{\lambda}(\IHat = 1) \leq \delta$.  
		Since $\tau_\delta$ is a stopping time, the event $C_\delta \cap (\IHat = 1)$ belongs to $\cF_{n_\delta}$. Hence, by Inequality~\eqref{StrongConverse}, for all $x \in \R$, 
		\[\delta \geq e^{-x}\left[\bP_\mu \left(C_\delta,\IHat = 1\right) - \bP_{\mu}\left(L_{n_\delta}(\mu,\epsilon) \geq x\right)\right]\;,\]
		which can be rewritten as 
		\[\bP_\mu \left(C_\delta,\IHat = 1\right) \leq \delta e^{x} + \bP_{\mu}\left(L_{n_\delta}(\mu,\epsilon) \geq x\right)\;.\]
		The choice $x = (1-\beta/2) \ln (1/\delta)$ yields 
		\begin{equation}\bP_\mu \left(C_\delta,\IHat = 1\right) \leq \delta^{\frac{\beta}{2}} + \bP_{\mu}\left(\frac{L_{n_\delta}(\mu,\epsilon)}{n_\delta} \geq \frac{1-\beta/2}{1-\beta} \frac{\left(|\mu| + \epsilon\right)^2}{2\sigma^2}\right)\;.\label{eq:Interm1}\end{equation}
		By the law of large numbers and the fact that $ n_{\delta} \to + \infty$ when $\delta\rightarrow 0$, 
		\[\frac{L_{n_\delta}(\mu,\epsilon)}{n_\delta} \stackrel{\delta\to 0}{\longrightarrow} \bE_\mu\left[ \ln \frac{\ell(X_1;\mu)}{\ell(X_1;\epsilon)} \right] = \mathrm{KL}\left(\bP_\mu^{X_1},\bP_\epsilon^{X_1}\right) = \frac{(\mu - \epsilon)^2}{2\sigma^2} < \frac{1-\beta/2}{1-\beta} \frac{\big(|\mu| + \epsilon\big)^2}{2\sigma^2\;}\;,\]
		and the RHS of Equation~\eqref{eq:Interm1} goes to zero. 
		
		\paragraph{$\bm{\lim_{\delta \rightarrow 0} \bP_{ \mu}\left(C_\delta, \IHat = 2\right) = 0}$}
		
		Using the exact same reasoning as above, this time with $\lambda = - \epsilon$ which is such that $\bP_\lambda(C_\delta,\IHat = 2) \leq \delta$, yields 
		\begin{equation}\bP_\mu \left(C_\delta,\IHat = 2\right) \leq \delta^{\frac{\beta}{2}} + \bP_{\mu}\left(\frac{L_{n_\delta}(\mu,-\epsilon)}{n_\delta} \geq \frac{1-\beta/2}{1-\beta} \frac{\left(|\mu| + \epsilon\right)^2}{2\sigma^2}\right)\;.\label{eq:Interm2}\end{equation}
		Writing, thanks to the law of large numbers,
		\[\frac{L_{n_\delta}(\mu,-\epsilon)}{n_\delta} \stackrel{\delta\to 0}{\longrightarrow} \frac{(\mu + \epsilon)^2}{2\sigma^2} < \frac{1-\beta/2}{1-\beta} \frac{\left(|\mu| + \epsilon\right)^2}{2\sigma^2}\;,\]
		proves that the RHS in \eqref{eq:Interm2} goes to zero. 
		
	\end{proof}

	\section{A $\delta$-CORRECT PARALLEL GLRT IN A BANDIT MODEL} \label{sec:PACbandit}
	
	In this section we formally introduce the multi-armed bandit model, under which we can study a broad family of \emph{active} tests of possibly overlapping hypotheses.   
	
	In this work, we consider bandit models where all arms belong to a canonical one-parameter exponential family
	\[\cF = \big\{ \nu_\theta \ \text{of density }  f_\theta(x) = \exp(\theta x - b(\theta)) \ \text{w.r.t.} \ \xi ; \;\;\theta \in \Theta\big\}\;,\]
	where $\xi$ is some reference measure, $\Theta\subseteq \R$ is an interval of natural parameters and $b : \Theta \rightarrow \R$ is convex and twice differentiable. For appropriate choices of $b$ and $\xi$, the class $\cF$ can be the set of Bernoulli distributions, Gaussian distribution with a known variance, Poisson, Exponential distributions, etc. We refer the reader to \citep{KLUCBJournal} for more details about the use of these exponential families for bandit problems. We just recall that distributions in a one-parameter exponential family can be alternatively parameterized by their means $\mu =\dot{b}(\theta)$. Hence, if $\cI = \dot{b}\left(\Theta\right)$, an exponential bandit model can be parameterized by $\bm \mu = (\mu_1,\dots,\mu_K) \in \cI^K$. Letting $\nu^{\mu}$ be the unique distribution that has mean $\mu\in \cI$, the Kullback-Leibler divergence between the distribution of mean $\mu$ and than of mean $\mu'$ is
	\[d(\mu,\mu') := \text{KL}\left(\nu^{\mu},\nu^{\mu'}\right) = {\mu}\left(\dot{b}^{-1}({\mu}) - \dot{b}^{-1}(\mu')\right) - b\left(\dot{b}^{-1}({\mu})\right) + b\left(\dot{b}^{-1}(\mu')\right)\;.\]
	Remarkable special cases are the family of Gaussian distributions with variance $\sigma^2$, with $d(\mu,\mu')=(\mu - \mu')^2/(2\sigma^2)$, and the family of Bernoulli distributions associated to the binary relative entropy $d(\mu,\mu') = \mu\ln\frac{\mu}{\mu'} + (1-\mu)\ln\frac{1-\mu}{1-\mu'}$ .
	
	The accumulation of observations in a bandit model is governed by a \emph{sampling rule} which specifies, at each round $t$, the arm $A_t$ that is selected; a sample $X_t$ from the distribution $\nu^{\mu_{A_t}}$ is subsequently collected. The sampling rule is sequential in that $A_t$ is only allowed to depend on the past observations $A_1,X_1,\dots,A_{t-1},X_{t-1}$ and possibly on some exogenous (independent) source of randomness.
	
	\paragraph{Fixed sampling rule: A sequential testing problem} 
	We consider $M$ regions  $\cR_1,\dots,\cR_M$ covering $\cI^K$: $\cR = \bigcup_{i=1}^M \cR_i = \cI^K$ (but that do not necessarily form a partition).
	For any given sampling rule, we are interested in building a sequential test $({\tau}_\delta,\IHat)$ that $\delta$-correctly tests the $M$ possibly overlapping hypotheses 
	\begin{equation}\cH_1 : \left(\bm\mu \in \cR_1\right) \ \ \cH_2 : \left(\bm\mu \in \cR_2\right) \ \ \dots \ \ \cH_M : \left(\bm\mu \in \cR_M\right)\label{TheTest}\end{equation}
	based on the stream of observations $A_1,X_1,A_2,X_2, \dots$. We denote by $N_a(t) = \sum_{s=1}^t\ind_{(A_s = a)}$ the number of selections of arm $a$ up to round $t$ and use the notation $\hat{\mu}_a(t) = \frac{1}{N_a(t)}\sum_{s=1}^t X_s\ind_{(A_s = a)}$ for the empirical mean of those observations.  
	
	Letting $\bm{\hat{\mu}}(t) = (\hat{\mu}_1(t),\dots, \hat{\mu}_K(t))$ be the vector of empirical means of the arms, for every $\bm{\lambda}=(\lambda_1,\dots,\lambda_K)\in\cR$ the log-likelihood ratio can be expressed in the following way:
	\begin{align}
	\ln \frac{\ell(X_1,\dots,X_t ; \bm{\hat{\mu}}(t))}{\ell(X_1,\dots,X_t ; \bm\lambda)} & =  \sum_{a=1}^KN_a(t) \Big[ \hat{\mu}_a(t)\left(\dot{b}^{-1}(\hat{\mu}_a(t)) - \dot{b}^{-1}(\lambda_a)\right) - b\left(\dot{b}^{-1}(\hat{\mu}_a(t))\right) + b\left(\dot{b}^{-1}(\lambda_a)\right)\Big]\nonumber\\
	& =  \sum_{a=1}^K N_a(t)d\left(\hat{\mu}_a(t),\lambda_a\right)\;.\label{eq:llrformula}
	\end{align}
	Hence, the parallel GLRT for testing \eqref{TheTest} is defined by
	\begin{eqnarray}\tau_\delta &=& \inf \left\{ t \in \N : \max_{a = 1,\dots,M} \inf_{\lambda \in \cR \backslash \cR_a} \sum_{j=1}^K N_j(t) d\left(\hat{\mu}_j(t),\lambda_j\right) > \beta(t,\delta)\right\}\;, \label{ParticularGLRT}\\
	\IHat &=& \underset{a = 1,\dots,M}{\text{argmax}} \ \inf_{\lambda \in \cR \backslash \cR_a} \sum_{j=1}^K N_j(\tau_\delta) d\left(\hat{\mu}_j(\tau_\delta),\lambda_j\right)\;.\nonumber 
	\end{eqnarray}
	In Lemma~\ref{lem:deltaCorrect} below, we provide a choice of the threshold function $\beta(t,\delta)$ for which the parallel GLRT is $\delta$-correct, whatever the sampling rule. The definition of $\beta$ requires to introduce the function 
	\begin{equation}\label{eq:T}
	\cT(x) ~=~ 2 \tilde h\left(\frac{h^{-1}(1+x) + \ln(2\zeta(2))}{2}\right)
	\end{equation}
	where, for $u \ge 1$, $h(u) = u - \ln u$ and for any $x \ge 0$ 
	\[
	\tilde h(x)
	~=~
	\begin{cases}
	e^{1/h^{-1}(x)} h^{-1}(x) & \text{if $x \ge h^{-1}(1/\ln (3/2))$,}
	\\
	(3/2) (x-\ln \ln (3/2))
	& \text{otherwise.}
	\end{cases}
	\]
	The function $\cT$ is easy to compute numerically. Its use for the construction of anytime, self-normalized confidence intervals is detailed in~\citep{KK18Mixture}, where the following approximations are derived: 
	$\cT (x) \simeq x + 4 \ln\big(1 + x + \sqrt{2x}\big)$ for $x\geq 5$ and $\cT(x) \sim x$ when $x$ is large. It can be noted that in the particular case of Gaussian bandits Lemma~\ref{lem:deltaCorrect} holds for a slightly smaller function $\cT$ that satisfies $\cT (x) \simeq x + \ln(x)$.

	\begin{lemma}\label{lem:deltaCorrect} For any sampling rule, the parallel GLRT test $(\tau_\delta,\IHat)$ using the threshold function  
		\[\beta(t,\delta) = 3 K \ln(1 + \ln t) + K\cT\left(\frac{\ln(1/\delta)}{K}\right)\label{UniversalThreshold}\]
		is $\delta$-correct: for all $\bm\mu \in \cR$, $\bP_{\bm\mu}\left(\tau_\delta < \infty, \bm\mu \notin \cR_{\IHat} \right) \leq \delta$.
	\end{lemma}

	\begin{proof} Let $\bm\mu \in \cR$. The probability of error  $\bP_{\bm\mu}\left(\tau_\delta < \infty, \bm\mu \notin \cR_{\IHat} \right)$ can be upper-bounded by
		\begin{align*}
		& \bP_{\bm\mu}\left(\exists t, \exists i :\bm\mu \notin \cR_i , \inf_{\bm \lambda \in \cR \backslash \cR_i} \ln \frac{\ell(X_1,\dots,X_t ; \hat{\bm\mu}(t))}{\ell(X_1,\dots,X_t ;\bm \lambda)} > \beta(t,\delta)\right)\\
		& \leq  \bP_{\bm\mu}\left(\exists t, \exists i : \bm\mu \in \cR\backslash\cR_i  , \ \ln \frac{\ell(X_1,\dots,X_t ; \hat{\bm\mu}(t))}{\ell(X_1,\dots,X_t ;\bm \mu)} > \beta(t,\delta)\right) \\
		& \leq  \bP_{\bm\mu}\left(\exists t, \ln \frac{\ell(X_1,\dots,X_t ; \hat{\bm\mu}(t))}{\ell(X_1,\dots,X_t ;\bm \mu)} > \beta(t,\delta)\right)\;.
		\end{align*}
		Hence, to prove the $\delta$-correctness of the parallel GLRT, it is sufficient to control the deviations of the quantity  $\ln \frac{\ell(X_1,\dots,X_t ; \hat{\bm\mu}(t))}{\ell(X_1,\dots,X_t ;\bm \mu)}$ uniformly over time. For this purpose, in the case of an exponential family bandit model, one can use the deviations bounds recently obtained by~\cite{KK18Mixture}. More precisely, thanks to Equation~\eqref{eq:llrformula},
		\begin{align*}
		& \bP_{\bm\mu}\left(\tau_\delta < \infty, \bm\mu \notin \cR_{\IHat} \right)  \leq  \bP_{\bm\mu}\left(\exists t \in \N : \sum_{a=1}^K N_a(t)d(\hat{\mu}_a(t),\mu_a) > \beta(t,\delta)\right) \\
		& \leq  \bP_{\bm\mu}\left(\exists t \in \N : \!\sum_{a=1}^K \! N_a(t)d(\hat{\mu}_a(t),\mu_a) \! > 3\!\sum_{a=1}^K \!\ln(1+\ln N_a(t)) + K\cT\left(\frac{\ln(1/\delta)}{K}\right)\right),
		\end{align*}
		which is upper bounded by $\delta$ according to~Theorem 14 in \citep{KK18Mixture}. In the Gaussian case, Corollary 10 in the same paper justifies the use of a slightly smaller function~$\cT$.  
	\end{proof}
	
	\begin{remark}\label{rem:Gaussian} The sequential testing problem studied in Section~\ref{sec:twogaussians} can be viewed as a one-armed bandit problem with a Gaussian arm, for which $d(x,y) = (x-y)^2/(2\sigma^2)$. It is  easy to check that the expression in \eqref{ParticularGLRT} coincides with that in \eqref{strategy} for the test studied in Section~\ref{sec:twogaussians}. Therefore, Lemma~\ref{lem:deltaCorrect} can be applied with $K=1$ to justify the chosen threshold.  
	\end{remark}

	\paragraph{Active testing: Optimizing the test duration} 
	Lemma~\ref{lem:deltaCorrect} provides a universal choice of threshold for which the parallel GLRT is $\delta$-correct, whatever the sampling rule.
	Obviously, however, in order to minimize the number of samples $\tau_\delta$ needed before stopping the test, the sampling rule also plays a crucial role. For example, if for $K=3$ the question is whether $\mu_2 > \mu_3$, there is no point in sampling from $\nu_{1}$. 
	
	In an \emph{active} identification problem, the goal is precisely to optimize the stopping rule \emph{and} the sampling rule in order to reach the minimal number of samples, while preserving the $\delta$-correctness of the decision. The sample complexity of some active identification problems has been exhibited when the different regions form a \emph{partition} \citep{GK16,KKG18Murphy,KK18Mixture,Juneja19} but the case of overlapping hypotheses is known to be more complex \citep{Degenne18Sticky}. 
	
	In the rest of the paper, we focus on a particular example of active identification: $\epsilon$-Best Arm Identification for which we investigate the optimal sample complexity. As we show below, a \emph{lower-bound} analysis of the sample complexity indirectly leads to an efficient sampling rule.

	\section{SAMPLE COMPLEXITY LOWER BOUNDS FOR ACTIVE TESTING: $\epsilon$-BEST ARM IDENTIFICATION} \label{sec:lbPACBAI}
	
	$\epsilon$-Best-Arm Identification (BAI) was first introduced by \cite{MannorTsi04,EvenDaral06} in the bandit literature. Given a risk parameter $\delta \in (0,1]$ and an accuracy parameter $\epsilon>0$, the goal is to find as quickly as possible an element of the set $\cA_{\epsilon}(\bm\mu) := \{ a: \mu_a \geq \max_{i} \mu_i - \epsilon\}$ of $\epsilon$-optimal arms, with probability at least $1-\delta$, by adaptively sampling the arms. A slightly different formulation with an \emph{indifference zone} has been studied in the ranking and selection community (see, e.g., \cite{KimNelson01}) in which correctness guarantees are provided only when there is a unique $\epsilon$-best arm (i.e. for $\bm\mu$ such that $|\cA_{\epsilon}(\bm\mu)| = 1$). In contrast, we highlight that in $\epsilon$-BAI, in the presence of multiple $\epsilon$-best arms (i.e., overlapping hypotheses), our goal is to identify one of those. 
	
	A possible motivation to study $\epsilon$-best arm identification stems from its application to A/B/C Testing: In this context, each arm models some version of a company's website. For each visitor, the company adaptively selects one of the versions and checks whether a conversion occurs. The goal is to identify one version with the highest probability of conversion. In the presence of many near optimal versions, it makes sense to relax the problem of finding a near-optimal version, as it leads to shorter tests.    
	
	$\epsilon$-best arm identification can be regarded as a particular example of \emph{active testing} of the hypotheses \[\cH_1 = \left(\mu_1 > \max_{i} \mu_i - \epsilon\right) \ \ \dots \ \  \cH_K = \left(\mu_K > \max_i \mu_i - \epsilon\right)\]
	based on adaptively sampling the marginals of $f_{\bm\mu} = f_{\mu_1} \otimes \dots \otimes f_{\mu_K}$. 
	A strategy for $\epsilon$-BAI not only consists in a sequential test for those hypotheses, that is a pair $(\tau,\hat{\imath})$ of stopping and recommendation rule, but also on a sampling rule $(A_t)_t$ indicating which arm is selected at round $t$. As explained in Section~\ref{sec:PACbandit}, we consider the case where the arms belongs to a one-dimension canonical exponential family and their distribution is therefore parameterized by their means that belong to some interval $\cI = (\mu^-, \mu^+)$. 
	
	\begin{definition} A $(\epsilon,\delta)$-PAC strategy for $\epsilon$-best arm identification is a triple $(A_t^\delta,\tau_\delta,\IHat)$, where $A_t^\delta$ is $\mathcal{F}_{t-1}$-measurable for each $t\geq 1$, $\tau$ is a stopping time relative to the filtration  $(\mathcal{F}_t)_t$, and $\IHat$ is a $\mathcal{F}_\tau$-measurable random variable, such that
		\[\forall \bm \mu \in \cI^K, \ \ \ \bP_{\bm\mu}\big(\tau_\delta < \infty, \IHat \notin \cA_\epsilon(\bm\mu)\big)\leq \delta\;.\] 
		An $\epsilon$-PAC strategy is a family of strategies $\pi_\delta = (A_t^\delta,\tau_\delta,\IHat)$ such that, for all $\delta$, $\pi_\delta$ is $(\epsilon,\delta)$-PAC.
	\end{definition}
	
	The goal in $\epsilon$-BAI is to build an $(\epsilon,\delta)$-PAC strategy such that on every bandit instance $\bm\mu$, the expected number of samples needed before stopping $\bE_{\bm\mu}[\tau_\delta]$, is as close as possible to the \emph{sample complexity}, which is the minimal number of samples needed by any ($\epsilon,\delta$)-PAC algorithm before stopping. In this section, we provide lower bounds on the sample complexity of $\epsilon$-Best Arm Identification.  
	
	\subsection{The Characteristic Time of $\epsilon$-Best Arm Identification} All the lower bounds that we propose feature a quantity $T^*_\epsilon(\bm\mu)$ which we term the characteristic time of~$\mu$.

	\begin{definition} The characteristic time $T^*_\epsilon(\bm\mu)$ is defined by:
		\[T^*_\epsilon(\bm\mu)^{-1} = \sup_{w \in \Sigma_K} \max_{a \in \cA_\epsilon(\bm\mu)} \min_{b \neq a} \inf_{(\lambda_a,\lambda_b) : \lambda_a \leq \lambda_b - \epsilon} \Big[w_a d(\mu_a,\lambda_a) + w_b d(\mu_b,\lambda_b)\Big],\]
		where $\Sigma_K = \{w \in [0,1]^K : \sum_{i} w_i = 1\}$ is the set of probability vectors.
	\end{definition}
	
	It can be observed that the quantity $T^*_0(\bm\mu)$ coincides with the quantity $T^*(\bm\mu)$ introduced by \cite{GK16} as the characteristic time for the best arm identification problem. 
	
	In order to illustrate the dependency of $T_\epsilon^*(\bm\mu)$ on the means $\bm\mu$, we first consider the Gaussian case. Introducing the notation $\bm\mu^{a,\epsilon}$ for the bandit instance such that $\mu^{a,\epsilon}_i = \mu_i$ for all $i\neq a$ and $\mu^{a,\epsilon}_a = \mu_a + \epsilon$, it can be shown (see Appendix~\ref{proof:GaussianComputation}) that
	\begin{equation}T^*_\epsilon(\bm\mu) = \min_{a \in \cA_\epsilon(\bm\mu)} T_0\left(\bm\mu^{a,\epsilon}\right) = T_0\left(\bm\mu^{a^*,\epsilon}\right),\label{GaussianSolved}\end{equation}
	where $a^* \in \text{argmax}_a \mu_a$ is an optimal arm. Hence it is sufficient to be able to compute the characteristic time for $\epsilon=0$, for which \cite{GK16} provide an efficient algorithm. Besides, \cite{GK16} propose an approximation for $T_0^*(\bm\mu)$ from which it directly follows that
	\[\sum_{a\neq a^*}^K \frac{2\sigma^2}{(\mu_{a^*}+\epsilon - \mu_a)^2}\leq T^*_\epsilon(\bm\mu) \leq 2 \times \sum_{a\neq a^*}^K \frac{2\sigma^2}{(\mu_{a^*}+\epsilon - \mu_a)^2}\;.\]
	Finally, there is a simple expression in the two armed-case: $T^*_\epsilon(\bm \mu) = \frac{8\sigma^2}{(|\mu_1 - \mu_2| + \epsilon)^2}$. 
	Explicit expressions for general (non-Gaussian) two-armed bandits can be found in Section~\ref{subsec:TwoArmsImproved}. However, beyond 2 arms there is no closed-form expression for the characteristic time $T^*_\epsilon(\bm\mu)$ and we refer to Section~\ref{sec:OptimalStrategy} for a discussion on its efficient computation. 
	
	\subsection{A First Non-Asymptotic Lower Bound}\label{subsec:ComplexityTerm}
	
	To illustrate the difficulty of extending the lower bound technique of~\cite{GK16} to $\epsilon$-BAI, we begin with a sub-optimal but simple sample complexity lower bound, obtained using the high-level change of distribution argument presented in Lemma~\ref{lem:CD}. 
	
	\begin{theorem}\label{thm:LBGeneral} For any $(\epsilon,\delta)$-PAC strategy, one has 
		\[\bE_{\bm\mu}[\tau_\delta] \geq  {T^*_\epsilon(\bm\mu)}\left[\frac{1-\delta}{|\cA_{\epsilon}(\bm\mu)|}\ln\left(\frac{1}{\delta}\right) - \ln(2)\right]\;.\]
	\end{theorem}
	
	\begin{proof} Using the $\delta$-PAC property and the pigeonhole principle, there exists $a\in\cA_{\epsilon}(\bm \mu)$ such that \[\bP_{\bm\mu}(\IHat = a) \geq \frac{1-\delta}{|\cA_{\epsilon}(\bm\mu)|}.\] 
		By definition of $T^*_\epsilon(\bm\mu)$, for the above choice of $a\in \cA_{\epsilon}(\bm \mu)$, there exists ${b}\neq a$ and $\overline{\bm\lambda}$ such that $\overline{\lambda}_a < \overline{\lambda}_b - \epsilon$:  
		\begin{equation}{T^*_\epsilon(\bm\mu)^{-1}} \geq \frac{\bE_{\bm\mu}[N_a(\tau_\delta)]}{\bE_{\bm\mu}[\tau_\delta]}d(\mu_a,\overline{\lambda}_a) + \frac{\bE_{\bm \mu}[N_b(\tau_\delta)]}{\bE_{\bm\mu}[\tau_\delta]}d(\mu_b,\overline{\lambda}_b)\;.\label{eq:StillTrue}\end{equation}
		Indeed, fixing $\overline{w}_a := \bE_\mu[N_a(\tau_\delta)] / \bE_\mu[\tau_\delta]$ and defining 
		\[F_a(w,\bm\mu) = \min_{b \neq a} \inf_{\lambda_a < \lambda_b - \epsilon} [w_a d(\mu_a,\lambda_a) + w_b d(\mu_b,\lambda_b)]\;,\]
		observe that 
		\[\sup_{w \in \Sigma_K} F_a(w,\bm \mu) \geq F_a(\overline{w}, \bm\mu)\;.\]
		Letting $b, \overline{\lambda}_a,\overline{\lambda}_b$ be the minimizer in the definition of $F_a(\overline{w},\bm\mu)$ yields 
		\[\max_{a \in \cA_\epsilon}\sup_{w \in \Sigma_K} F_a(w,\bm \mu) \geq \overline{w}_a d(\mu_a,\overline{\lambda}_a) + \overline{w}_b d(\mu_b,\overline{\lambda}_b)\;,\]
		and the left hand side of this inequality is $T_\epsilon^{-1}(\bm\mu)$. 
		
		We now assume $\bE[\tau_\delta] < \infty$ (otherwise any lower bound on the number of samples of that algorithm obviously holds). Using Wald's Lemma, the expression of the expected log-likelihood ratio between $\bm\mu$ and $\overline{\bm \lambda}$ is 
		\[\bE_{\bm\mu}\left[L_{\tau_\delta}(\bm\mu,\overline{\bm \lambda})\right] =\bE_{\bm\mu}[N_a(\tau_\delta)]d(\mu_a,\overline{\lambda}_a) + \bE_{\bm \mu}[N_b(\tau_\delta)]d(\mu_b,\overline{\lambda}_b)\;.\]
		Therefore, one can use Inequality~\eqref{eq:StillTrue} and the high-level form of Lemma~\ref{lem:CD} to write
		\begin{align*}\frac{\bE_{\bm\mu}[\tau_\delta]}{T^*_\epsilon(\bm\mu)} &\geq \bE_{\bm\mu}\left[L_{\tau_\delta}(\bm\mu,\overline{\bm \lambda})\right] \;\geq \; \mathrm{kl}(\bP_{\bm\mu}(\IHat = a), \bP_{\overline{\bm \lambda}}(\IHat = a))\\
		&\geq \mathrm{kl}\left(\frac{1-\delta}{|\cA_{\epsilon}(\bm\mu)|} ,\delta\right)  \;\geq \; \frac{1-\delta}{|\cA_{\epsilon}|}\ln\left(\frac{1}{\delta}\right) - \ln(2)\;,\end{align*}
		where the last inequality follows from the inequality $\mathrm{kl}(x,y) \geq x\ln(1/y) - \ln(2)$. 
	\end{proof}
	
	\subsection{A Tighter Lower Bound for Converging Strategies}
	
	The lower bound obtained in Theorem~\ref{thm:LBGeneral} could incorrectly suggest that the sample complexity would be of order  $\frac{T^*_\epsilon(\bm\mu)}{|\cA_\epsilon(\bm\mu)|} \ln\left(\frac{1}{\delta}\right)$. As we will see, the factor $|\cA_\epsilon(\bm\mu)|$ in the denominator is not correct: this bound is tight  only for values of $\bm\mu$ with a unique $\epsilon$-best arm. We now provide a tighter lower bound, which comes with two shortcomings. First, it is asymptotic: it relies on the low-level form of change of distribution presented in Lemma~\ref{lem:CD}. Second, it requires an additional assumption on the sampling rule, that needs to be \emph{converging}. 
	
	\begin{definition}\label{def:Converging}A sampling rule $A_t^\delta$ is said to be \emph{converging} if it does not depend on $\delta$ (that is, $A_t^\delta=A_t$) and for every $\bm \mu$, there exists $\overline{w}(\bm\mu) \in \Sigma_K$ such that,  $\bP_{\bm\mu}$- almost surely, 
		\[\forall a= 1, \dots,K, \ \ \ \frac{N_a(t)}{t} \underset{t\rightarrow \infty}{\longrightarrow} \overline{w}_a(\bm \mu)\;.\]
	\end{definition}
	
	\begin{theorem}\label{thm:LBConvergent} For any $\epsilon$-PAC family of converging strategies, one has 
		\[\liminf_{\delta \rightarrow 0}\frac{\bE_{\bm \mu}[\tau_\delta]}{\ln(1/\delta)} \geq T^*_\epsilon(\bm \mu)\;.\] 
	\end{theorem}
	
	\begin{proof} To prove Theorem~\ref{thm:LBConvergent}, it is sufficient to prove that for every $\beta > 0$, 
		\begin{equation}
		\lim_{\delta \rightarrow 0} \bP\left(\tau_\delta \leq (1-\beta) T^*(\bm\mu)\ln\left(\frac{1}{\delta}\right) \right) = 0\;.
		\label{ToProve2}
		\end{equation}
		Indeed, the conclusion follows by Markov inequality and letting $\beta$ go to zero. We now fix $\beta > 0$ and introduce 
		
		\vspace{-0.3cm}
		
		\[n_\delta : = \left\lfloor (1-\beta) T^*_\epsilon(\bm\mu) \ln\left(\frac{1}{\delta}\right)\right\rfloor\;.\]
		Defining $\cC_\delta =(\tau_\delta \leq n_\delta)$, we establish that $\lim_{\delta \rightarrow 0} \bP_{\bm\mu}(C_\delta) = 0$.
		
		First, using that the strategy is $\delta$-correct, one can write 
		\begin{eqnarray*}\bP_{\bm\mu}(C_\delta) &=& \bP_{\bm\mu}\left(C_\delta, \IHat \in \cA_{\epsilon}(\bm\mu)\right) + \bP\left( C_\delta, \IHat \notin \cA_\epsilon(\bm\mu)\right) \\
			& \leq & \sum_{a \in \cA_\epsilon(\bm\mu)}\bP_{\bm\mu}(C_\delta, \IHat = a) + \delta\;.
		\end{eqnarray*}
		To conclude, it is sufficient to prove that $\forall a \in \cA_{\epsilon}(\bm\mu)$, $\lim_{\delta\rightarrow 0} \bP_{\bm\mu}(C_\delta, \IHat = a) = 0$.
		
		Fix $a \in \cA_{\epsilon}(\bm\mu)$. Just like in the proof of Theorem~\ref{thm:LBGeneral}, we use that by definition of $T_\epsilon^*(\bm\mu)$, there exists $b\neq a$ and an alternative model $\bm\lambda$ in which only arm $a$ and $b$ are changed and such that $\lambda_{a} <\lambda_{b} - \epsilon$ that satisfies 
		\begin{equation}T_\epsilon^*(\bm \mu)^{-1} \geq \overline{w}_a(\bm\mu) d\left(\mu_a,\lambda_a\right) + \overline{w}_b(\bm\mu) d\left(\mu_a,\lambda_b\right)\;, \label{SpecificChoice}\end{equation}
		where we recall that the weights vector $\overline{w}(\bm\mu)$ contains the limit fraction of samples allocated to each arm by the converging sampling rule $(A_t)$: for all $i$, it holds that $\bP_{\bm\mu}\left(\lim_{t\rightarrow \infty}N_i(t) / t = \overline{w}_i(\bm\mu)\right)=1$. 
		
		As arm $a$ is not $\epsilon$-optimal in the bandit instance $\bm\lambda$, one has $\bP_{\lambda}(C_\delta,\IHat = a) \leq \bP_{\lambda}(\IHat = a) \leq \delta$. Moreover, the event $C_\delta \cap (\IHat = a)$ belongs to $\cF_{n_\delta}$. Hence, using the low-level Inequality~\eqref{StrongConverse} in Lemma~\ref{lem:CD}, for all $x \in \R$, 
		\[\delta \geq e^{-x}\left[\bP_{\bm\mu} \left(C_\delta,\IHat = a\right) - \bP_{\bm\mu}\left(L_{n_\delta}(\bm\mu,\bm\lambda) \geq x\right)\right]\]
		which can be rewritten as
		\[\bP_{\bm\mu} \left(C_\delta,\IHat = a\right) \leq \delta e^{x} + \bP_{\bm\mu}\left(L_{n_\delta}(\bm\mu,\bm\lambda) \geq x\right)\;.\]
		The choice $x = (1-\beta/2) \ln (1/\delta)$ yields 
		\begin{equation}\bP_\mu \left(C_\delta,\IHat = a\right) \leq \delta^{\frac{\beta}{2}} + \bP_{\bm\mu}\left(\frac{L_{n_\delta}(\bm\mu,\bm\lambda)}{n_\delta} \geq \frac{1-\beta/2}{1-\beta} {T^*_\epsilon(\bm\mu)}^{-1}\right)\;.\label{eq:IntermN1}\end{equation}
		Because of the converging assumption, it holds $\bP_{\bm\mu}$-a.s. (independently of $\delta$) that 
		\[\frac{L_t(\bm\mu,\bm\lambda_a)}{t} \underset{t\rightarrow \infty}{\longrightarrow}  \overline{w}_a(\bm\mu) d\left(\mu_a,\lambda_a\right) + \overline{w}_b(\bm\mu) d\left(\mu_b,\lambda_b\right) \leq {T^*_\epsilon(\bm\mu)}^{-1} < \frac{1-\beta/2}{1-\beta} {T^*_\epsilon(\bm\mu)}^{-1}\;,\]
		where the first inequality holds by definition of $\bm\lambda$, which satisfies \eqref{SpecificChoice}. As $n_\delta \rightarrow \infty$ when $\delta\rightarrow 0$, it follows that
		\[\lim_{\delta\rightarrow 0} \bP_{\bm\mu}\left(\frac{L_{n_\delta}(\bm\mu,\bm\lambda)}{n_\delta} \geq \frac{1-\beta/2}{1-\beta} {T^*_\epsilon(\bm\mu)}^{-1}\right) = 0\]
		and $\lim_{\delta \rightarrow 0} \bP_\mu \left(C_\delta,\IHat = a\right) = 0$, which concludes the proof.
	\end{proof}
	
	\subsection{An improved lower bound for two arms}\label{subsec:TwoArmsImproved} $\epsilon$-best arm identification in a two-armed bandit model $\bm\mu = (\mu_1,\mu_2)$ coincides with an active testing procedure to select one of the hypotheses
	\[\cH_1 : (\mu_1 > \mu_2 - \epsilon) \ \ \ \text{or} \ \ \ \cH_2 : (\mu_2 > \mu_1 - \epsilon)\]
	and an $(\epsilon,\delta)$-PAC strategy should satisfy  $\bP\left(\IHat = 2\right) \leq \delta$ for all $\bm\mu$ such that $\mu_1 > \mu_2 + \epsilon$ and $\bP\left(\IHat = 1\right) \leq \delta$ for all $\bm\mu$ such that $\mu_2 > \mu_1 + \epsilon$. In that case, a first observation is that the complexity term $T^*_{\epsilon}(\bm\mu)$ can be made slightly more explicit. The proof of Lemma~\ref{prop:AlternativeComplexity} can be found in Appendix~\ref{proof:Complexity2arms}.
	
	\begin{lemma}\label{prop:AlternativeComplexity} If $\mu_a \geq \mu_b - \epsilon$, define $\mu_{\epsilon}^*(\mu_a,\mu_b)$ as the unique solution in $\lambda \in (\mu^-,\mu^+ - \epsilon)$ to
		\[d(\mu_a,\lambda) = d(\mu_b,\lambda + \epsilon)\;.\]
		Then
		\[T^*_\epsilon(\bm \mu)^{-1} := \left\{\begin{array}{lcl}
		d(\mu_1,\mu_{\epsilon}^*(\mu_1,\mu_2)) & \text{if} & \mu_1 > \mu_2 + \epsilon \;, \\
		\max[d(\mu_1,\mu_{\epsilon}^*(\mu_1,\mu_2)),d(\mu_2,\mu_{\epsilon}^*(\mu_2,\mu_1))] & \text{if} & \mu_1 - \mu_2 \in (-\epsilon,\epsilon)\;,\\
		d(\mu_2,\mu_{\epsilon}^*(\mu_2,\mu_1)) & \text{if} & \mu_2 > \mu_1 + \epsilon\;.
		\end{array}
		\right.\]
	\end{lemma}

	Note that the quantity $\mu^{*}_\epsilon(\mu_a,\mu_b)$ defined in Lemma~\ref{prop:AlternativeComplexity} is indeed well-defined when $\mu_a \geq \mu_b - \epsilon$, as the the mapping $g: \mu\mapsto d(\mu_a,\mu)-d(\mu_b,\mu+\epsilon)$ is decreasing on $[\mu_b-\epsilon,\mu_a]$ and satisfies $g(\mu_b-\epsilon)=d(\mu_a,\mu_b-\epsilon)>0$ and $g(\mu_a) = - d(\mu_b,\mu_a+\epsilon) <0$.

	For two armed bandits, we provide below an asymptotic sample complexity lower bound that relax the converging assumption required for Theorem~\ref{thm:LBConvergent}. Theorem~\ref{thm:2armsGeneralSampling} given below still requires the sampling rule $A_t^\delta$ to be independent from $\delta$ (we write $A_t$ instead of $A_t^\delta$), but doesn't need the fraction of selection of each arm to converge to a certain proportion under that sampling rule. Its proof is given in Appendix~\ref{proof:2armsGeneralSampling}.
	
	\begin{theorem}\label{thm:2armsGeneralSampling} Every $\epsilon$-PAC family of \emph{anytime} strategies (i.e., for which the sampling rule $(A_t)$ is independent from $\delta$) satisfies, for all $\bm \mu$,  
		\[\liminf_{\delta \rightarrow 0} \frac{\bE_{\bm \mu}[\tau_\delta]}{\ln(1/\delta)} \geq T^*_\epsilon(\bm \mu)\;.\] 
	\end{theorem}

	\section{AYMPTOTICALLY OPTIMAL STRATEGIES FOR $\epsilon$-BEST ARM IDENTIFICATION}\label{sec:OptimalStrategy}
	
	The strategy described in this section hinges on the same ideas than the Track-and-Stop strategy proposed by \cite{GK16} for the case $\epsilon=0$. The sampling rule builds on the knowledge of ``optimal weights'' under which the arms should be sampled, whereas the stopping is a parallel GLRT test with an appropriate choice of threshold. We prove that the resulting strategy, called $\epsilon$-Track-and-Stop is asymptotically optimal for some instances of $\epsilon$-best arm identification. 
	
	\subsection{Optimal Weights and their Computation} \label{subsec:WeightsComputation} From the proof of our lower bounds, it appears that a converging strategy matching the lower bound should have its empirical proportions of draws converge to a vector that belongs to the set 
	\[\cW^*_{\epsilon}(\bm\mu)  =  \argmax{w \in \Sigma_K} \max_{a \in \cA_\epsilon} \min_{b \neq a} \inf_{(\lambda_a,\lambda_b) : \lambda_a \leq \lambda_b - \epsilon} \left[w_a d(\mu_a,\lambda_a) + w_b d(\mu_b,\lambda_b)\right]\;.
	\]
	To ease the notation, we assume that $\mu_1$ is an optimal arm, i.e. $\mu_1 \geq \mu_b$ for all $b\in \{1,\dots,K\}$. We define
	\[T_\epsilon^{*,a}(\bm\mu)^{-1} := \sup_{w \in \Sigma_K} \min_{b \neq a} \inf_{(\lambda_a,\lambda_b) : \lambda_a \leq \lambda_b - \epsilon} \left[w_ad(\mu_a,\lambda_a) + w_bd(\mu_b,\lambda_b)\right]\;.\]
	Observe that for $a$ such that there exists $b\neq a$ with $\mu_a = \mu_b-\epsilon$, then $T_\epsilon^{*,a}(\bm\mu) = +\infty$. For $a$ such that $\mu_a > \mu_b - \epsilon$ for all $b\neq a$, we argue below that the argmax in $w$ is unique and define 
	\begin{equation}
	w_\epsilon^{*,a}(\bm\mu) :=  \argmax{w \in \Sigma_K} \min_{b \neq a} \inf_{(\lambda_a,\lambda_b) : \lambda_a \leq \lambda_b - \epsilon} \left[w_ad(\mu_a,\lambda_a) + w_bd(\mu_b,\lambda_b)\right]\;. \label{def:WeightA}
	\end{equation}
	With this notation, for $\epsilon > 0$, one has
	\begin{equation}\cW_\epsilon^*(\bm \mu) = \left\{ w_\epsilon^{*,a}(\bm\mu)  \ \text{ for } \ a \in \argmin{i \in \cA_\epsilon(\bm\mu) : \mu_i > \max_{j\neq i} \mu_j - \epsilon} \ T_\epsilon^{*,i}(\bm\mu)\right\}\;.\label{def:SetWeights}\end{equation}
	For $\epsilon>0$, this set if always non-empty as arm 1 satisfy $\mu_1>\min_{j\neq 1}\mu_i-\epsilon$. For $\epsilon =0$, in the presence of a unique optimal arm, $\cW_\epsilon(\bm\mu) = \{w_0^{*,1}(\bm\mu)\}$, but in the presence of multiple arms with mean $\mu_1$, the optimal weights are not well defined. In this degenerate case, we define $\cW_\epsilon(\bm\mu) = \{|\cA_0(\bm\mu)|^{-1} \ind_{\cA_0(\bm\mu)}\}$.
	
	We now fix $a \in \cA_\epsilon(\bm\mu)$ such that $\mu_a > \max_{b\neq a}\mu_b - \epsilon$ and explain why $w_{\epsilon}^{*,a}(\bm \mu)$ is well-defined and how to compute this probability vector. Introducing the interval \[\cI_{a,b}^{\epsilon} = [\mu^- \vee (\mu_b-\epsilon) , \mu_a \wedge (\mu^+-\epsilon)]\;,\]  it follows from Lemma~\ref{prop:ComputingStuff} stated in Appendix~\ref{appendix:Computations} that  
	\begin{eqnarray*}w^{*,a}_{\epsilon}(\bm\mu) &=& \argmax{w \in \Sigma_K} \min_{b \neq a} \inf_{\lambda \in \cI_{a,b}^{\epsilon}} \left[w_ad(\mu_a,\lambda) + w_bd(\mu_b,\lambda+\epsilon)\right] \\
		& = & \argmax{w \in \Sigma_K} \min_{b \neq a} w_a g_b\left(\frac{w_b}{w_a}\right) 
	\end{eqnarray*}
	where for $b\neq a$ the application $g_b$ (also depending on $a$, $\bm\mu$ and $\epsilon$) is defined by 
	\begin{eqnarray}g_b(x) & = & \inf_{\lambda \in \cI_{a,b}^{\epsilon}} \left[d(\mu_a, \lambda) + x d(\mu_b, \lambda + \epsilon)\right]\;.\label{def:gb}\end{eqnarray}
	In the particular case in which $\mu_a > \mu^+ - \epsilon$ and $\mu_b = \mu^+$, the interval $\cI_{a,b}^{\epsilon}$ is reduced to the point $\mu^+ - \epsilon$ and the mapping $g_b(x) = d(\mu_a,\mu^+ - \epsilon)$ is constant. Besides this degenerate case, $\cI_{a,b}^{\epsilon}$ is a non-degenerate compact interval, and one can leverage this property to show that the mapping $x \mapsto g_b(x)$ is increasing. Moreover, it satisfies $g_b(0)=d(\mu_a,\mu_a \wedge (\mu^+-\epsilon))$ and $\lim_{x\rightarrow \infty} g_b(x) = d(\mu_a,(\mu_b-\epsilon)\vee \mu^-)$. 
	Therefore, one can define the inverse mapping \[x_b(y) = g_b^{-1}(y) \ \ \text{for} \ \ y \in \big[d(\mu_a,\mu_a \wedge (\mu^+-\epsilon)),d(\mu_a,(\mu_b-\epsilon)\vee \mu^-)\big).\]
	We furthermore define the constant mapping $x_a(y) = 1$ for all  $y \in \big[d(\mu_a,\mu_a \wedge (\mu^+-\epsilon)),d(\mu_a,(\max_{b\neq a} \mu_b - \epsilon)\vee\mu^-)\big)$.
	
	To state Theorem~\ref{prop:OptimalWeights} below, which provides a more explicit expression of the weights $w_\epsilon^{*,a}(\bm\mu)$, we also introduce the notation $\lambda_b(x)$ for the minimizer in~\eqref{def:gb}:  
	\[g_b(x) = d\big(\mu_a, \lambda_b(x)) + x d(\mu_b, \lambda_b(x) + \epsilon\big)\]
	The proof of Theorem~\ref{prop:OptimalWeights}, given in Appendix~\ref{proof:OptimalWeights}, hinges on the fact that at the optimum all the values of $g_b(w_b/w_a)$ are identical, which allows to re-parameterize the optimization problem by their common value.
	
	\begin{theorem}\label{prop:OptimalWeights} Fix a bandit instance $\bm\mu$ and $a \in \cA_\epsilon(\bm\mu)$ and let 
		\[\cI_a = \Big[d(\mu_a,\mu_a \wedge (\mu^+-\epsilon)),d(\mu_a,(\max_{b\neq a}\mu_b-\epsilon))\vee \mu^-\Big)\;.\]
		\begin{enumerate}
			\item For degenerate instances $(\bm\mu,\epsilon)$ such that $\mu_a > \mu^+ - \epsilon$ and there exists $b\neq a$ such that $\mu_b=\mu^+$, \[(w^{*,a}_\epsilon(\bm\mu))_b = \ind_{(b = a)}\;.\]
			\item For other instances $(\bm\mu,\epsilon)$, let $y^*$ be the unique solution on $\cI_a$ to the equation 
			\begin{equation}\sum_{b\neq a} \frac{d(\mu_a,\lambda_b(x_b(y)))}{d(\mu_b,\lambda_b(x_b(y)) + \epsilon)}=1\;.\label{eq:ComputeY}\end{equation}
			Then the argmax in \eqref{def:WeightA} is unique and is defined by 
			\[\left(w_\epsilon^{*,a}(\bm\mu)\right)_b = \frac{x_b(y^*)}{\sum_{b=1}^k x_b(y^*)}.\]
		\end{enumerate}
	\end{theorem}
	
	In practice, the optimal weights can be computed by solving \eqref{eq:ComputeY} using binary search. Each computation of the left-hand side requires to compute $\lambda_b(x_b(y))$. For each $x$, $\lambda_b(x)$ is the minimizer of a twice differentiable convex function and can be approximated numerically using for example Newton's method. In the Gaussian case, we may also use the closed form $\lambda_b(x) = \frac{\mu_a + x(\mu_b-\epsilon)}{1+x}$. Then the value of the inverse function $x_b(y)$ can be obtained by solving $g_b(x) = y$ using binary search. 
	
	\subsection{The $\epsilon$-Tracking rule} The Tracking sampling rule was originally proposed by \cite{GK16} for best arm identification ($\epsilon =0$), for which the set of optimal weights $\cW_0(\bm\mu)$ is reduced to a single vector, $w^*(\bm\mu)$. Building on our ability to compute those optimal weights, the Tracking rule is a mechanism that forces the empirical proportions of arm selections to converge to this oracle allocation, by relying on plug-in estimates $w^*(\hat{\bm\mu}(t))$ and ensuring enough exploration.
	
	We present below an extension of this rule, called $\epsilon$-Tracking, that can be used in any $\epsilon$-best arm identification problem. The difference with Tracking is that ties need to be broken when $\cW_*(\hat{\bm\mu}(t))$ contains multiple vectors. We do by relying on a fixed ordering of the arms, introducing the notation 
	\[w^*_\epsilon(\bm\mu) = w_{\epsilon}^{*,a^*}(\bm\mu) \ \ \text{ where} \ \ a^* = \min \Big\{ a \in \{1,\dots,K\} : w_\epsilon^{*,a}(\bm\mu) \in \cW_\epsilon(\bm\mu)\Big\}.\]
	At round $t+1$, $\epsilon$-Tracking first computes the set $U_t = \{a : N_a(t) < \sqrt{t} - K/2\}$ of arms that are currently under-sampled. Then it selects 
	\[A_{t+1} \in \left\{\begin{array}{ll}
	\underset{a \in U_t}{\text{argmin}}  \ N_a(t) \quad \text{if} \ U_t \neq \emptyset & (\textit{forced  exploration}), \text {or otherwise}\\
	\underset{1 \leq a \leq K}{\text{argmax}} \big[ t \times(w_\epsilon^*)_a(\hat{\bm\mu}(t)) -{N_a}(t)\big] & (\textit{tracking the plug-in estimate}).
	\end{array}
	\right.\]
	Observe that this sampling rule does not depend on $\delta$. Other Tracking rules exist, for example tracking the cumulative proportions (C-Tracking), which forces $N_a(t)$ to be close to $\sum_{s=1}^t w^*_a(\hat{\bm\mu}_s(t))$. Their relative merits is discussed in \citep{Degenne18Sticky}.

	\begin{definition}\label{def:regular} An instance $(\bm\mu,\epsilon)$ is regular if the set $\cW_\epsilon^*(\bm \mu)$ is of cardinality one. For a regular instance, $\cW_\epsilon^*(\bm\mu) =\left\{w_\epsilon^*(\bm\mu)\right\}$.  
	\end{definition}
	
	Using the same arguments as those used by \cite{GK16} for the Tracking rule, one can show that for regular instances $\epsilon$-Tracking is a converging strategy (in the sense of Definition~\ref{def:Converging}), which converges to the optimal proportions. 
	
	\begin{lemma}\label{lem:Converging} Let $(\bm\mu,\epsilon)$ be a regular instance of almost optimal best arm identification. Then, under the $\epsilon$-Tracking sampling rule, 
		\[\bP_{\bm\mu}\left(\forall a =1,\dots,K, \ \ \lim_{t\rightarrow\infty} \frac{N_a(t)}{t} = \left(w_\epsilon^*(\bm\mu)\right)_a\right)= 1\;.\] 
	\end{lemma}

	\subsection{The Parallel GLRT Stopping Rule for $\epsilon$-Best Arm Identification}
	
	Introducing for all $a,b$ the statistic
	\[\hat{Z}^\epsilon_{a,b}(t) = \inf_{\bm\lambda : \lambda_a < \lambda_b - \epsilon} \big[ N_a(t)d(\hat{\mu}_a(t),\lambda_a) + N_b(t)d(\hat{\mu}_b(t),\lambda_b)\big]\]
	and $\hat{\cA}_{\epsilon}(t) = \{ a : \hat{\mu}_a(t) \geq \max_{j}\hat{\mu}_j(t) - \epsilon\}$, the parallel GLRT stopping rule can be written 
	\begin{equation}\tau_\delta  =   \inf \left\{ t \in \N : \max_{a \in \hat{\cA}_\epsilon(t)} \min_{b \neq a} \hat{Z}^\epsilon_{a,b}(t)  > \beta(t,\delta)\right\}
	\label{PGLRTeps} \end{equation}
	and the associated recommendation rule is 
	\begin{equation}\IHat \in \underset{{a \in \hat{\cA}_\epsilon(\tau_\delta)}}{\text{argmax}} \ \min_{b \neq a} \hat{Z}^\epsilon_{a,b}(\tau_\delta).\label{RecommendationRule}\end{equation} 
	
	In the Gaussian case, a more explicit expression of $\hat{Z}^\epsilon_{a,b}(t)$ can be given for $a$ and $b$ such that $\hat{\mu}_a(t) \geq \hat{\mu}_b(t) - \epsilon$ (which holds when $a \in \hat{\cA}_\epsilon(t)$): 
	\[\hat{Z}^\epsilon_{a,b}(t) = \frac{1}{2\sigma^2}\frac{N_a(t)N_b(t)}{N_a(t) + N_b(t)}\big(|\hat{\mu}_a(t) - \hat{\mu}_b(t)| + \epsilon\big)^2\;,\]
	Beyond the Gaussian case, when $\hat{\mu}_a(t) \geq \hat{\mu}_b(t) - \epsilon$ one can write 
	\begin{eqnarray*}\hat{Z}^\epsilon_{a,b}(t) &=& \inf_{\lambda \in ]\mu^-,\mu^+-\epsilon[} \big[N_a(t)d(\hat{\mu}_a(t),\lambda) + N_b(t)d(\hat{\mu}_b(t),\lambda+\epsilon)\big]\;,\end{eqnarray*}
	and the solution of this minimization problem has to be numerically approximated. 
	
	Lemma~\ref{lem:deltaCorrect} ensures that the parallel GLRT stopping rule defined in \eqref{PGLRTeps} based on the threshold 
	\[\beta(t,\delta) = 3 K \ln(1 + \ln t) + K\cT\left(\frac{\ln(1/\delta)}{K}\right)\]
	yields an $(\epsilon,\delta)$-PAC strategy when combined with any sampling rule. Using a more refined analysis tailored to particular structure of the parallel GLRT test for $\epsilon$-best arm identification permit to guarantee the $(\epsilon,\delta)$-PAC property for a smaller threshold. Note that the threshold $\beta(t,\delta) = \log\left(\frac{2(K-1)t}{\delta}\right)$ proposed by \cite{GK16} for $\epsilon =0$ (and that rely on different deviation inequalities) can also be used here, but the one given in Lemma~\ref{lem:PGLRTcorrect} can be smaller for large values of $t$.
	
	\begin{lemma}\label{lem:PGLRTcorrect} For any sampling rule, the parallel GLRT stopping rule with threshold  
		\begin{equation}\beta(t,\delta) = 6 \ln(1 + \ln t) + 2\cT\left(\frac{\ln((K-1)/\delta)}{2}\right)\label{RefinedThreshold}\end{equation}
		satisfies $\bP\left(\tau_\delta < \infty, \IHat\notin \cA_{\epsilon}\right) \leq \delta$.
	\end{lemma}
	
	\begin{proof}
		
		\begin{align*}
		\bP_{\bm\mu}\left( \mu_{\hat{a}_\tau} < \mu_1 - \epsilon\right)  &\leq  \sum_{b \notin \cA_\epsilon} \bP_{\bm\mu} (\hat{a}_\tau = b, \tau < \infty) \\
		& \leq  \sum_{b \notin \cA_\epsilon} \bP_{\bm\mu}\left(\exists t\in \N: \inf_{ \mu_b' - \mu_1' < - \epsilon} \left[N_1(t) d(\hat{\mu}_1(t),\mu_1') + N_b(t) d(\hat{\mu}_b(t),\mu_b') \right] > \beta(t,\delta)\right) \\
		& \leq  \sum_{b \notin \cA_\epsilon} \bP_{\bm\mu}\left( \exists t \in \N: N_1(t) d(\hat{\mu}_1(t),\mu_1) + N_b(t) d(\hat{\mu}_b(t),\mu_b) > \beta(t,\delta)\right) \\
		& \leq  \sum_{b \notin \cA_\epsilon} \frac{\delta}{K-1} = \delta\;, 
		\end{align*}
		where the last inequality relies on Theorem 14 in \citep{KK18Mixture}.
	\end{proof}

	\subsection{The $\epsilon$-Track-and-Stop Strategy} We define the $\epsilon$-Track-and-Stop strategy with parameter $\delta$ as the following strategy for $\epsilon$-Best Arm Identification: 
	\begin{itemize}
		\item its sampling rule is the $\epsilon$-Tracking rule
		\item its stopping rule is the parallel GLRT with threshold $\beta(t,\delta)$ given in~\eqref{RefinedThreshold}
		\item its recommendation rule is the one associated to the parallel GLRT, that is $\hat \imath_{\tau_\delta}$ defined in \eqref{RecommendationRule}.
	\end{itemize}
	This strategy, that we denote by $\epsilon$-TaS($\delta$) has the following properties. 
	
	\begin{theorem}\label{thm:OptimalRegular} Fix $\delta \in (0,1]$. Then $\epsilon$-TaS($\delta$) is $(\epsilon,\delta)$-PAC. Moreover, for every \emph{regular} instance $(\bm\mu,\epsilon)$, if $\tau_\delta$ denotes the stopping rule of $\epsilon$-TaS($\delta$),
		\[\limsup_{\delta\rightarrow 0}\frac{\bE_{\bm\mu}[\tau_\delta]}{\ln(1/\delta)} \leq T^*_\epsilon(\bm\mu)\;.\]
	\end{theorem}
	
	This result shows that the family of $\epsilon$-TaS($\delta$) strategies for $\delta$ in $(0,1)$, which is a family of converging strategies by Lemma~\ref{lem:Converging}, is $\epsilon$-PAC and is matching the lower bound given in Theorem~\ref{thm:LBConvergent}, for regular instances. 
	
	Beyond regular instances, the picture is intriguingly more complex, and was recently studied in details by \cite{Degenne18Sticky}, who consider a general active testing framework with overlapping hypotheses. They show that $\epsilon$-Tracking can fail to converge to an element in $\cW_\epsilon^*(\bm\mu)$ and propose a ``sticky'' variant fixing this problem. This alternative sampling rule turns out not to be a converging strategy anymore, as under two different runs it may converge to a different element in $\cW_\epsilon^*(\bm\mu)$. However, the authors manage to extend the lower bound of Theorem~\ref{thm:LBConvergent} to \emph{any} strategy (using a nice game-theoretic argument). Hence the quantity $T^*_\epsilon(\bm\mu)$ is indeed the right characteristic number of samples for $\epsilon$-best arm identification. 
	
	\paragraph{Sketch of proof} To understand why the parallel GLRT attains the lower bound of Theorem~\ref{thm:LBConvergent} if coupled with a good sampling rule, we first rewrite its stopping rule 
	\[\tau_\delta = \inf\left\{ t \in \N^* : \hat{Z}(t) \geq \beta(t,\delta)\right\}\;,\]
	where we introduce
	\begin{align*}&\hat Z(t)  =  \max_{a \in \hat{\cA}_\epsilon(t)} \min_{b\neq a} \inf_{\lambda \in (\mu^-,\mu^+-\epsilon)} \left[N_a(t) d (\hat{\mu}_a(t),\lambda) + N_b(t)d(\hat{\mu}_b(t),\lambda + \epsilon)\right]\\
	& \hspace{0.3cm} =  t \times \left(\max_{a \in \hat{\cA}_\epsilon(t)} \min_{b\neq a} \inf_{\lambda \in (\mu^-,\mu^+-\epsilon)} \left[\frac{N_a(t)}{t} d (\hat{\mu}_a(t),\lambda) + \frac{N_b(t)}{t}d(\hat{\mu}_b(t),\lambda + \epsilon)\right]\right)\;.
	\end{align*}
	By Lemma~\ref{lem:Converging}, under a regular instance for which $\cW^*_\epsilon(\bm\mu) = \{w_\epsilon^*(\bm\mu)\}$ the $\epsilon$-Tracking sampling rule satisfies that, for every arm $a$, $\hat{\mu}_a(t)$ converges almost surely to $\mu_a$ and the sampling proportion $N_a(t)/t$ converges almost surely to the target $(w_\epsilon^*(\bm\mu))_a$. Using the fact that $\bm\mu \mapsto w_\epsilon^*(\bm\mu)$ is continuous in any regular instance $\bm\mu$, for large values of $t$ it holds that
	\begin{align*}\hat Z(t) & \simeq  t \times \left(\max_{a \in {\cA}_\epsilon} \min_{b\neq a} \inf_{\lambda \in (\mu^-,\mu^+-\epsilon)} \left[(w_\epsilon^*(\bm\mu))_a d ({\mu}_a,\lambda) + (w_\epsilon^*(\bm\mu))_bd({\mu}_b,\lambda + \epsilon)\right]\right) \\
	& =  t \times (T^*_\epsilon(\bm\mu))^{-1}\;.
	\end{align*}
	Therefore one obtains \begin{equation}\tau_\delta \simeq \inf\left\{ t\in \N^* : t \geq T^*_\epsilon(\bm\mu) \beta(t,\delta)\right\}\;.\label{Quantity}\end{equation}
	If the threshold function were chosen to be $\beta(t,\delta) = \ln(1/\delta)$ one would immediately get the upper bound $T^*_\epsilon(\bm\mu) \ln(1/\delta)$. However, a slightly larger threshold is needed to ensure the $(\epsilon,\delta)$-PAC property of $\epsilon$-TaS($\delta$). Still, the threshold defined in \eqref{RefinedThreshold} satisfies  $\beta(t,\delta) \leq \ln(1/\delta) + \ln(t)$ for large values of $t$ hence Lemma~\ref{lem:Aurelien} can be applied to show that the right hand side of \eqref{Quantity} can be upper bounded by\[T^*_\epsilon(\bm\mu)\ln\left(\frac{1}{\delta}\right) + 2T_\epsilon^*(\bm\mu) \ln\left(T_\epsilon^*(\bm\mu)\ln\left(\frac{1}{\delta}\right)\right)\;.\]
	This argument can be made rigorous to prove that on every regular instance $\bm\mu$,
	\[\bP_{\bm\mu}\left(\lim_{\delta \rightarrow 0} \frac{\tau_\delta}{\ln(1/\delta)} \leq T^*_\epsilon(\bm\mu)\right) = 1\;, \]
	in the spirit of the proof of Proposition 13 in \citep{GK16}. 
	The proof of Theorem~\ref{thm:OptimalRegular} is a bit more complex as it requires to control the expectation $\bE[\tau_\delta]$. It can be obtained by following similar arguments as the proof of Theorem 14 in \citep{GK16}.

	\section{NUMERICAL EXPERIMENTS} 
	
	We report results of some experiments comparing $\epsilon$-Track-and-Stop (TaS) to some state-of-the-art algorithms for $\epsilon$-best arm identification for Bernoulli distributed arms. These competitors are KL-LUCB, KL-Racing \citep{COLT13} and UGapE \citep{Gabillon12UGapE}, briefly described below. For instances having a unique optimal arm, we also include $0$-TaS in our study (designed for finding the unique best arm) in order to assess the gain in sample complexity achieved when studying the $\epsilon$ relaxation. 
	
	We experiments with two regular instances, 
	\[
	\begin{array}{cclrcl}
	\bm\mu_1 &=& [0.2 \ \ 0.4 \ \ 0.5 \ \  0.55 \ \ 0.7] & \epsilon_1 &= &0.1 \\
	\bm\mu_2 &=& [0.4 \ \ 0.5 \ \  0.6 \ \  0.7 \ \  0.75 \ \ 0.8]& \epsilon_2 &=&0.15\;,\\
	\end{array}
	\]
	and one non-regular instance $\bm\mu_3 = [0.2 \ \ 0.3 \ \ 0.45 \ \ 0.55 \ \ 0.6 \ \ 0.6]  \ \ \ \epsilon_3=0.1\;$.

	In the experiments we set $\delta = 0.1$ and perform $N=1000$ independent replications in order to estimate the probability of error $\bP_{\bm\mu_i}(\bm\mu_{\hat\imath_{\tau_\delta}} \notin \cA_{\epsilon_i}(\bm\mu_i))$, and the expected number of samples $\bE_{\bm\mu_i}[\tau_\delta]$ for each algorithm and problem instance. First, Table~\ref{tab:Reco} permits to check that all algorithms are indeed $\delta$-correct: their empirical probability of error is (much) smaller than the maximal value $\delta=0.1$ which is prescribed. 
	
		\begin{table}[h]
		\centering
		\begin{tabular}{|c|c|c|c|c|c|}
			\hline
			&  $\epsilon$-TaS & KL-LUCB &  UGapE & KL-Racing & 0-TaS \\
			\hline
			$\bm\mu_1$, $\epsilon_1=0.1$  & 0.015 & 0.003   & 0.002 & 0.003 & 0.003 \\ \hline
                $\bm \mu_2$, $\epsilon_2 = 0.15$ & 0 & 0 & 0 & 0 & 0\\ \hline
			$\bm \mu_3$, $\epsilon_3 = 0.1$ & 0.001 & 0 & 0 & 0 &  -\\
			\hline
		\end{tabular}
		
		\smallskip
		
		\caption{\label{tab:Reco}\label{tab:results} Empirical error probabilities for various algorithms estimated on $N=1000$ repetitions.
		}
	\end{table}
	
	In those experiments, the stopping rule of Track-and-Stop is the parallel GLRT with threshold $\beta(t,\delta) = \ln((1+\ln(t))/\delta)$. This threshold is an approximation of the theoretical threshold given in Lemma~\ref{lem:PGLRTcorrect} that we recommend to use in practice. All the other algorithms rely on confidence bounds, and we implement them with upper and lower confidence bounds  based on the Kullback-Leibler divergence
	\begin{eqnarray*}
		u_a(t) & = & \max \left\{ q \in [0,1] : N_a(t) d(\hat{\mu}_a(t),q) \leq \beta(t,\delta) \right\} \ \ \text{and} \\
		\text{and} \ \ \ell_a(t) & = & \min \ \left\{ q \in [0,1] : N_a(t) d(\hat{\mu}_a(t),q) \leq \beta(t,\delta) \right\}\;,
	\end{eqnarray*}
	for the same choice of $\beta(t,\delta)=\ln((1+\ln(t))/\delta)$. The stopping rule of KL-LUCB and UGapE consists in waiting for the lower confidence bound of the candidate best arm to overlap by at most $\epsilon$ with the upper-confidence bound of all other arms. KL-Racing is an algorithm based on eliminations that also exploits those confidence intervals: arms in a list of candidate arms are selected in a round robin fashion and when some arm $a$ satisfies $\ell_b(t) > u_a(t) - \epsilon$ for some other arm $b$, it is removed from the list of candidate arms. The algorithm stops when all candidate arms are empirically $\epsilon$-optimal. Those confidence-based stopping or elimination rules seem to be far too conservative compared to the parallel GLRT stopping rule, as the sample complexity results now reveal.

	\begin{table}[h]
		\centering
		\begin{tabular}{|c|c|c|c|c|c|c|}
			\hline
			& $T^*(\bm\mu) \ln(1/\delta)$ & $\epsilon$-TaS & KL-LUCB &  UGapE & KL-Racing & 0-TaS \\
			\hline
			$\bm\mu_1$, $\epsilon_1=0.1$ & 97&  171 \ (104)& 322  \ (137) & 324  \ (143)& 372 \ (159)& 433 \ (277) \\ 
			\hline
			$\bm \mu_2$, $\epsilon_2 = 0.15$ & 108 & 162 \ (83) & 345 \  (135)& 344 \ (141)  & 402 \ (146) &  2659 \ (1863)\\ 
			\hline
			$\bm \mu_3$, $\epsilon_3 = 0.1$ & 531 & 501 \ (261) & 1236 \ (403) & 1199 \  (414) & 1348 \ (436)&  -\\
			\hline
		\end{tabular}
		
		\smallskip
		
		\caption{\label{tab:SC}\label{tab:results2} Estimated values of $\bE_{\bm\mu_i}[\tau_\delta]$ based on $N=1000$ repetitions for different instances and algorithms (standard deviation indicated in parenthesis).
		}
	\end{table}

	Indeed, it can be seen in Table~\ref{tab:SC} that $\epsilon$-TaS requires about half the number of samples needed by other methods before stopping, and is sometimes getting very close to the ideal sample complexity $T^*(\bm\mu)\ln(1/\delta)$. In the first instance, in which there is a single $\epsilon$-best arm, one can see that $\epsilon$-TaS slightly improves over $0$-TaS due to the $\epsilon$-relaxation. This improvement is much more flagrant in the second instance, which contains 3 $\epsilon$-best arms.

	For the two instances with 3 $\epsilon$-best arms $\bm\mu_2$ and $\bm\mu_3$, we now report in Table~\ref{tab:Allocation} the optimal weights in $\cW_\epsilon(\bm\mu)$. There is one such weight for $\bm\mu_2$ and two for $\bm\mu_3$ in which the fraction allocation to the two best arms are exchanged. We also report in this table the empirical fraction of selection of each arm by $\epsilon$-TaS.

	\begin{table}[h]
		\centering
		\begin{tabular}{|c|c|c|c|c|c|c|}
			\hline
			$(\bm\mu_2)_a$ & 0.4 & 0.5 & 0.6 & \textbf{0.7} & \textbf{0.75} & \textbf{0.8} \\ \hline
			$\bm \left(w^*_{\epsilon_2}(\bm\mu_2)\right)_a$ & 0.024 & 0.036 & 0.060 & 0.136 & 0.275 & 0.469 \\ \hline
			$\bE_{\bm\mu_2}[N_a(\tau_\delta)/\tau_\delta]$ & 
			0.079 & 0.077 & 0.099 & 0.156 & 0.235 & 0.353 \\
			\hline
		\end{tabular}
		
		\vspace{0.2cm}
		
		\begin{tabular}{|c|c|c|c|c|c|c|}
			\hline
			$(\bm\mu_3)_a$ & 0.2 & 0.3 & 0.45 & \textbf{0.55} & \textbf{0.6} & \textbf{0.6} \\ \hline
			$\bm \left(w^{*,5}_{\epsilon_3}(\bm\mu_3)\right)_a$ & 0.008 & 0.0133 & 0.035 & 0.114 & 0.436 & 0.393\\ \hline
			$\bm \left(w^{*,4}_{\epsilon_3}(\bm\mu_3)\right)_a$ & 0.008 & 0.0133 & 0.035 & 0.114  & 0.393 & 0.436\\ \hline
			$\bE_{\bm\mu_3}[N_a(\tau_\delta)/\tau_\delta]$ & 0.049 & 0.049 & 0.081 & 0.194 & 0.317 & 0.309 \\ \hline
		\end{tabular}
		\caption{\label{tab:Allocation} Optimal proportions and empirical proportions achieved by $\epsilon$-TaS ($\epsilon$-optimal arms in bold)}
	\end{table}

	For the regular instance $\bm\mu_2$, according to Lemma~\ref{lem:Converging} the empirical proportions converge to the true proportions. However, as can be seen from Table~\ref{tab:Allocation} the convergence is slow and did not yet occur at the stopping time $\tau_\delta$ (all sub-optimal arms are slightly over sampled which may be a consequence of the forced exploration scheme). For the non-regular instance $\bm\mu_3$ in which there are two candidate optimal weights, $\epsilon$-TaS has no convergence guarantee towards one of these weights (and \cite{Degenne18Sticky} propose a fix). Just like for $\bm\mu_2$, we see that all sub-optimal arms are slightly over-sampled under $\bm\mu_3$. On average over many simulations, the strategy spends about the same time on the two optimal arms, which could indicate that ``convergence'' occurs towards $w^{*,4}_{\epsilon_3}(\bm\mu_3)$ or $w^{*,5}_{\epsilon_3}(\bm\mu_3)$ half of the time. Even without proper convergence guarantees, $\epsilon$-TaS performs very well on $\bm\mu_3$ as can be noted from Table~\ref{tab:SC}.

	\section{CONCLUSION} 
	
	We proposed a non-asymptotic way to analyze parallel Generalized Likelihood Ratio Tests, which can be used is various adaptive testing settings, from sequential to active testing, even in the presence of overlapping hypotheses. We also presented a way to derive lower bounds providing the exact asymptotic sample complexity of some testing problems. Remarkably, for problem instances that belong to multiple hypotheses, non-asymptotic lower bounds are still out of reach, maybe because of the use of the low-level change of low formula. An interesting open question is thus to know whether lower bounds can always be derived using only the high-level form, in a way that provides non-asymptotic results. So far, we were not able to achieve this goal, nor to show that it is impossible. 
	
	Natural extensions of our approach can be obtained for other sequential and active testing problems. For example, in bandit models, our approach  is likely to apply also for thresholding bandits, or for best-arm identification with a structure on the set of arms. Moreover, if a regularity constraint like $\mu_{i+1}-\mu_i \leq L$ could be handled, this would open the way to the study of the sample complexity of optimizing a Lipschitz continuous function. In that context, one can hope to exhibit a optimal density of sampling  (corresponding to the optimal weights in the finite case), depending on the function only, which an optimal optimization procedure would need to track. All these developments are left for future investigations.

\section*{ACKNOWLEDGEMENTS}

Emilie Kaufmann acknowledges the support of the Agence Nationale de la Recherche (ANR) under grants ANR-16-CE40-0002 (BADASS project), ANR-19-CE23-0026-04 (BOLD project) and the European CHIST-ERA project DELTA. Aurélien Garivier acknowledges the support of the Project IDEXLYON of the University of Lyon, in the framework of the Programme Investissements d'Avenir (ANR-16-IDEX-0005).

\newpage

	\appendix 
	
	\section{NON-ASYMPTOTIC SAMPLE COMPLEXITY ANALYSIS FOR A GAUSSIAN TEST} \label{proof:GaussianNonAsymptotic}

	We provide a general sample complexity analysis of the parallel GLRT test using a generic threshold function $\beta(t,\delta)$. Introducing the function 
	\[f(t,\delta):=\frac{|\mu| + \epsilon}{\sigma}\sqrt{t} -\sqrt{2\beta(t,\delta)}\;,\]
	we make the following assumptions: 
	\begin{enumerate}
		\item $\beta(t,\delta)  \leq \ell(\delta) + c \ln(t)$ where $\ell(\delta) \geq \ln(1/\delta)$. This permits to define $T_0(\delta)$ to be any constant such that 
		\[\forall t \geq T_0(\delta), \ \  f(t,\delta) \geq 0\;.\]
		\item For all $\delta \in (0,1], \forall t \geq T_0(\delta),$
		\begin{equation}  \frac{\partial }{\partial t}(\sqrt{2 \beta(t,\delta)}) \leq  \frac{1}{2}\frac{\partial }{\partial t}\sqrt{t\frac{(|\mu| + \epsilon)^2}{\sigma^2}} =\frac{|\mu| + \epsilon}{4\sigma^2}\frac{1}{\sqrt{t}}\;.\label{ass:Beta}\end{equation}
	\end{enumerate}
	A consequence of Assumption~2 is that $t \mapsto f(t,\delta)$ is non-decreasing, which permits to upper bound the expectation of $T_\delta$ as follows, introducing $(Z_t)$ a sequence of standard normal random variables.
	\begin{align*}
	\bE[T_{\delta}] & \leq \sum_{t=1}^{\infty} \bP\left(t \frac{(|\hat{\mu}_t| + \epsilon)^2}{2\sigma^2} \leq \beta(t,\delta)\right) \\
	& \leq  \sum_{t=1}^{\infty}\bP\left(\sqrt{t} \left(\left|\mu + \frac{\sigma Z_t}{\sqrt{t}}\right| + \epsilon\right) \leq \sqrt{2\sigma^2\beta(t,\delta)}\right) \\
	& \leq  \sum_{t=1}^{\infty}\bP\left(\sqrt{t} \left(|\mu| - |\sigma Z_t / \sqrt{t}|+\epsilon\right) \leq \sqrt{2\sigma^2\beta(t,\delta)}\right) \\
	& \leq  \sum_{t=1}^{\infty}\bP\left(|Z_t| \geq  \frac{|\mu| + \epsilon}{\sigma}\sqrt{t} -\sqrt{2\beta(t,\delta)}\right) \\
	& \leq  \lceil T_0(\delta)\rceil + \sum_{t=\lceil T_0(\delta)\rceil + 1}^\infty\bP\left(|Z_t| \geq  f(t,\delta)\right)dt \\
	& \leq  \lceil T_0(\delta)\rceil + \int_{T_0(\delta)}^\infty\bP\left(|Z_t| \geq  f(t,\delta)\right) dt\;.
	\end{align*}
	Using assumption \eqref{ass:Beta} one can write 
	\begin{align*}
	f(t,\delta)&=   f(T_0(\delta),\delta) + \int_{T_0(\delta)}^tf'(s,\delta)ds \\
	& \geq  0 + \frac{1}{2}\int_{T_0(\delta)}^t \left(\sqrt{t\frac{(|\mu| + \epsilon)^2}{\sigma^2}}\right)'ds \\
	& = 
	\int_{T_0(\delta)}^t\frac{|\mu| + \epsilon}{4\sigma}\frac{1}{\sqrt{s}}ds\;. \\
	\end{align*}
	For $t \in [T_0(\delta),2T_0(\delta)]$, observe that \[f(t,\delta) \geq (t-T_0(\delta)) \frac{|\mu| + \epsilon}{4\sqrt{2}\sigma}\frac{1}{\sqrt{T_0(\delta)}}\;,\]
	while for $t > 2T_0(\delta)$, 
	\[f(t,\delta) \geq (t-T_0(\delta)) \frac{|\mu| + \epsilon}{4\sigma}\frac{1}{\sqrt{t}} \geq 
	\sqrt{t-T_0(\delta)} \frac{|\mu| + \epsilon}{4\sqrt{2}\sigma}\;.\]
	
	Hence 
	\begin{align*}
	&\bE[T_{\delta}]
	\leq \lceil T_0(\delta)\rceil + \int_{T_0(\delta)}^{2T_0(\delta)}2\overline{\Phi}(f(t,\delta))dt + \int_{2T_0(\delta)}^\infty2\overline{\Phi}(f(t,\delta))dt \\
	& \leq  \lceil T_0(\delta)\rceil + 2\int_{T_0(\delta)}^{\infty}\overline{\Phi}\left((t-T_0(\delta)) \frac{|\mu| + \epsilon}{4\sqrt{2}\sigma}\frac{1}{\sqrt{T_0(\delta)}}\right)dt + 2\int_{2T_0(\delta)}^\infty\overline{\Phi}\left(\sqrt{t-T_0(\delta)} \frac{|\mu| + \epsilon}{4\sqrt{2}\sigma}\right)dt \\
	& \leq  \lceil T_0(\delta)\rceil + \frac{8\sqrt{2}\sigma}{|\mu|+\epsilon}\sqrt{T_0(\delta)}\int_{0}^{\infty}\overline{\Phi}\left(u\right)du + \frac{64\sigma^2}{(|\mu|+\epsilon)^2}\int_{T_0(\delta)\frac{(|\mu|+\epsilon)^2}{8\sigma^2}}^\infty\overline{\Phi}\left(\sqrt{u}\right)du \\
	& \leq  \lceil T_0(\delta)\rceil + \frac{8\sqrt{2}\sigma}{|\mu|+\epsilon}\sqrt{T_0(\delta)} \frac{1}{\sqrt{2\pi}} + \frac{64\sigma^2}{(|\mu|+\epsilon)^2}\int_{T_0(\delta)\frac{(|\mu|+\epsilon)^2}{8\sigma^2}}^\infty e^{-\frac{u}{2}}du \\
	& \leq  1 + T_0(\delta) + \frac{8\sqrt{2}\sigma}{|\mu|+\epsilon}\sqrt{T_0(\delta)} \frac{1}{\sqrt{2\pi}} + \frac{64\sigma^2}{(|\mu|+\epsilon)^2} e^{-T_0(\delta)\frac{(|\mu|+\epsilon)^2}{8\sqrt{2}\sigma^2}}\;.
	\end{align*}
	If the exploration function $\beta$ satisfies 
	\[\beta(t,\delta) \leq \ell(\delta) + c\ln(t)\;,\]
	Lemma~\ref{lem:Aurelien} below permits to prove that one can pick
	\[T_0(\delta) = \frac{2\sigma^2}{(|\mu|+\epsilon)^2}\ell(\delta) + \frac{4c\sigma^2}{(|\mu|+\epsilon)^2} \ln\left(\frac{2\sigma^2}{(|\mu|+\epsilon)^2}\ell(\delta)\right)\;.\]
	In that case,
	\begin{multline*}
	\bE[T_\delta]  \leq  \frac{2\sigma^2}{(|\mu|+\epsilon)^2}\left[\ell(\delta) + 2c \ln\Bigg(\frac{2\sigma^2}{(|\mu|+\epsilon)^2}\ell(\delta)\right) \\+ 8\sqrt{\ell(\delta) + 2c \ln\left(\frac{2\sigma^2}{(|\mu|+\epsilon)^2}\ell(\delta)\right)}+ {32\delta^{1/8}}\Bigg]+ 1 \;. \\
	\end{multline*}
	Theorem~\ref{thm:GaussianNonAsymptotic} follows from observing that the threshold \eqref{ThresholdGLRTGaussian} for which the parallel GLRT is $\delta$-correct satisfies 
	\[\beta(t,\delta) \leq \cT\left(\ln(1/\delta)\right) + 3\frac{1 + \ln(t)}{e}\;,\]
	and thus Assumption 1 is satisfied with $c=3/e$ and $\ell(\delta) = \cT\left(\ln(1/\delta)\right) + 3/e$. It can  be checked by calculus that Assumption 2 is  also satisfied.

	\begin{lemma}\label{lem:Aurelien} Fix $\alpha\geq 0$ and $\gamma \geq 1 + \alpha$. Then for all $t > 0$,
		\[t \geq \gamma + 2\alpha \ln (\gamma) \ \ \Rightarrow \ \ \ t \geq \gamma + \alpha \ln(t)\;.\] 
	\end{lemma}
	
	\begin{proof}
		Let $f(t)=t-\gamma-\alpha\ln(t)$.
		Then $f'(t)=1-\alpha/t\geq 0$ for $t \geq \alpha$. Hence, for all $t\geq t_0:=\gamma+2\alpha\ln(\gamma)\geq \alpha$, 
		\begin{align*}
		f(t)\geq f(t_0)&=\gamma + 2\alpha\ln(\gamma)-\gamma - \alpha\ln\big(\gamma+ 2\alpha \ln(\gamma) \big)\\
		&=\alpha\ln(\gamma) - \alpha \ln\left(1+\frac{2\alpha}{\gamma}\ln(\gamma)\right)
		\end{align*}
		which has the same sign as
		\[\gamma-1-\frac{2\alpha}{\gamma}\ln(\gamma) \geq\gamma-1-\frac{2\alpha\gamma}{e\,\gamma} \geq \gamma-1-\alpha \geq 0\;.  \]
		
	\end{proof}
	
	\section{COMPUTING THE CHARACTERISTIC TIME}\label{appendix:Computations}
	
	The following result is useful at several places to provide more explicit expressions of the characteristic time $T^*_\epsilon(\bm\mu)$. It follows easily from monotonicity properties of $\lambda \mapsto d(\mu_a,\lambda)$ and $\lambda \mapsto d(\mu_b, \lambda)$. 
	
	\begin{lemma} \label{prop:ComputingStuff} Fix $w_a,w_b \in \R$ and $\mu_a$, $\mu_b$ such that $\mu_a \geq \mu_b - \epsilon$. 
		\begin{align*}
		&\inf_{\lambda_a < \lambda_b - \epsilon} \big[w_ad(\mu_a,\lambda_a) + w_bd(\mu_b,\lambda_b)\big]\\ 
		& \hspace{0.5cm}=\inf_{\lambda \in \big[\mu^-\vee (\mu_b-\epsilon),\mu_a \wedge(\mu^+-\epsilon)\big]} \big[w_ad(\mu_a,\lambda) + w_bd(\mu_b,\lambda+\epsilon)\big]\\
		& \hspace{0.5cm} =  \inf_{\lambda \in (\mu^-,\mu^+-\epsilon)} \big[w_ad(\mu_a,\lambda) + w_bd(\mu_b,\lambda+\epsilon)\big]\;.
		\end{align*}
	\end{lemma}
	\noindent This result follows from the monotonicity properties of the mapping $\lambda \mapsto d(\mu_a,\lambda)$ and $\lambda \mapsto d(\mu_b,\lambda)$ that are illustrated in Figure~\ref{fig:KL}. 
	
	\begin{figure}[h]
		\centering
		\includegraphics[height=6cm]{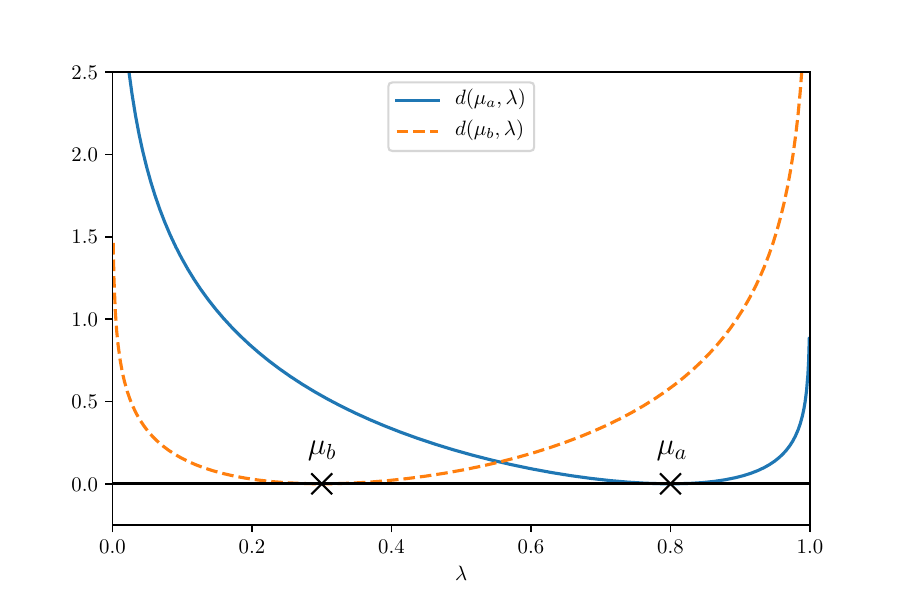}
		\caption{Monotonicity properties of the Kullback-Leibler divergence in the Bernoulli case \label{fig:KL}}
	\end{figure}

	First, we show that the minimizer in $(\lambda_a,\lambda_b)$ of 
	\[g(\lambda_a,\lambda_b) = w_ad(\mu_a,\lambda_a) + w_bd(\mu_b,\lambda_b)\]
	under the constraints $\lambda_a \leq \lambda_b - \epsilon$ is such that $\lambda_a \in \big[\mu^-\vee (\mu_b-\epsilon),\mu_a \wedge(\mu^+-\epsilon)\big]$. Indeed, if $\mu^- < \lambda_a < \mu_b - \epsilon$, 
	\[\inf_{\lambda : \lambda_b > \lambda_a + \epsilon} g(\lambda_a,\lambda) = g(\lambda_a,\mu_b   ) \geq g(\mu_b - \epsilon, \mu_b)\;.\]
	Similarly, if $\lambda_a > \mu_a$ (and $\lambda_a < \mu^+ - \epsilon$),
	\[\inf_{\lambda : \lambda > \lambda_a + \epsilon} g(\lambda_a,\lambda_b) = g(\lambda_a,\lambda_a + \epsilon) \geq g(\mu_a, \mu_a + \epsilon)\;.\]
	Now for $\lambda_a \in \big[\mu^-\vee (\mu_b-\epsilon),\mu_a \wedge(\mu^+-\epsilon)\big]$, $\lambda_b$ satisfying the constraint $\lambda_b \geq \lambda_a + \epsilon$ is larger than $\mu_b$. As $d(\mu_b,\lambda)$ is decreasing for $\lambda \geq \mu_b$, the value $\lambda_b$ that minimizes $w_bd(\mu_b,\lambda_b)$ is the smallest admissible value, that is $\lambda_b = \lambda_a + \epsilon$. This justifies the first equality in Lemma~\ref{prop:ComputingStuff}. 
	
	The second equality follows from the same monotonicity properties: letting $g(\lambda) = w_ad(\mu_a,\lambda) + w_bd(\mu_b,\lambda+\epsilon)$, one can easily show that the minimum of this function on $(\mu^-,\mu^+-\epsilon)$ is attained on $\big[\mu^-\vee (\mu_b-\epsilon),\mu_a \wedge(\mu^+-\epsilon)\big]$. Indeed, for all $\lambda < \mu_b - \epsilon$, $g(\lambda) > g(\mu_b-\epsilon)$ (as the two mappings $\lambda \mapsto d(\mu_a,\lambda)$ and $\lambda \mapsto d(\mu_b,\lambda+\epsilon)$ are decreasing for $\lambda < \mu_b - \epsilon$) and for all $\lambda > \mu_a$, $g(\lambda) > g(\mu_a)$ (the two mappings are increasing for $\lambda > \mu_a$).

	\subsection{Computing the Characteristic Time for Gaussian Distributions} \label{proof:GaussianComputation}
	
	In the Gaussian case, there is no edge effect as $\mu^- = -\infty$ and $\mu^+=+\infty$. Therefore one can write
	\begin{align*}
	T_\epsilon^*(\bm\mu)^{-1} & =  \sup_{w\in \Sigma_K}\max_{a \in \cA_\epsilon(\bm\mu)} \min_{b\neq a} \inf_{\lambda\in \R} \big[w_a d(\mu_a,\lambda) + w_b d(\mu_b,\lambda + \epsilon)\big]\\
	& =  \sup_{w\in \Sigma_K}\max_{a \in \cA_\epsilon(\bm\mu)} \min_{b\neq a} \inf_{\lambda\in \R} \big[w_a d(\mu_a,\lambda-\epsilon) + w_b d(\mu_b,\lambda)\big] \\
	& =  \max_{a \in \cA_\epsilon(\bm\mu)}\sup_{w\in \Sigma_K} \min_{b\neq a} \inf_{\lambda\in \R} \big[w_a d(\mu_a+\epsilon,\lambda) + w_b d(\mu_b,\lambda)\big]\\
	& =  \max_{a \in \cA_\epsilon(\bm\mu)} T^*_0\left(\bm\mu^{a,\epsilon}\right)^{-1}\;,
	\end{align*}
	where the last but one equality uses that $d(\mu_a,\lambda - \epsilon)=d(\mu_a+\epsilon,\lambda)$, which is a special feature of the divergence function $d(x,y) = (x-y)^2/(2\sigma^2)$. Recall that $\bm\mu^{a,\epsilon}$ is a bandit instance identical to $\bm\mu$ expect for arm $a$ that is set to $\mu_a + \epsilon$.
	
	Now, exploiting further the explicit form of the Gaussian divergence, one can write for any $\bm\lambda$ with a unique optimal arm $a^*=a^*(\bm\lambda)$  
	\[T^*_0(\bm\lambda)^{-1} = \sup_{w\in \Sigma_K} \inf_{b\neq a^*} \frac{w_{a^*}w_b}{w_{a^*}+w_b}\frac{\Delta_b(\bm\lambda)}{2\sigma^2} \ \ \ \ \text{with} \ \ \ \ \Delta_b(\bm\lambda) = \lambda^{a^*(\bm\lambda)} - \lambda_b\;.\]
	Moreover, if $\lambda$ does not have a unique optimal arm, $T_\epsilon(\bm\lambda)^{-1} = 0$. Letting 1 be an optimal arm in $\bm\mu$, it can be observed that for all $a$ such that $\mu_a > \mu_1 - \epsilon$, the instance $\bm\mu^{a,\epsilon}$ as a unique optimal arm and the vector of gaps $(\Delta_b(\bm\mu^{a,\epsilon}))_{b\neq a}$ is dominated by the vector of gaps $(\Delta_b(\bm\mu^{1,\epsilon}))_{b\neq 1}$, in the sense that there exists a permutation $\sigma$ such that $\Delta_b(\bm\mu^{a,\epsilon}) \leq \Delta_{\sigma(b)}(\bm\mu^{1,\epsilon})$ (intuitively, adding $\epsilon$ to $\mu_1$ creates the overall largest gaps in the bandit instance compared to adding it elsewhere). This yields $T_\epsilon(\bm\mu^{a,\epsilon})^{-1} \leq T_\epsilon(\bm\mu^{1,\epsilon})^{-1}$. For $a$ such that $\mu_a = \mu_1 - \epsilon$, $T_\epsilon(\bm\mu^{a,\epsilon})^{-1} = 0 \leq T_\epsilon(\bm\mu^{1,\epsilon})^{-1}$. 
	
	Putting things together, we proved that 
	\[T_\epsilon^*(\bm\mu) ^{-1} = \max_{a\in \cA_\epsilon(\bm\mu)} T_0^{*}(\bm\mu^{a,\epsilon})^{-1} =T_0^{*}(\bm\mu^{1,\epsilon})^{-1}\;,\]
	which easily yields the equality \eqref{GaussianSolved} stated in Section~\ref{subsec:ComplexityTerm}. 
	
	\subsection{Computing the Characteristic Time for Two-Armed bandits: proof of Lemma~\ref{prop:AlternativeComplexity}}\label{proof:Complexity2arms}
	
	The optimization problem defining $T^*_\epsilon(\bm\mu)^{-1}$ can rewritten 
	\begin{eqnarray}T^*_\epsilon(\bm\mu)^{-1} &=& \sup_{w \in \Sigma_K} \max_{a \in \cA_\epsilon(\bm\mu)} \min_{b \neq a} \inf_{(\lambda_a,\lambda_b) : \lambda_a \leq \lambda_b - \epsilon} \big[w_a d(\mu_a,\lambda_a) + w_b d(\mu_b,\lambda_b)\big]\nonumber\\
	&=& \sup_{w \in \Sigma_K} \max_{a \in \cA_\epsilon(\bm\mu)} \min_{b \neq a} \inf_{\lambda \in (\mu^-,\mu^+-\epsilon)} \big[w_ad(\mu_a,\lambda) + w_bd(\mu_b,\lambda+\epsilon)\big]\label{OptimRewriting}\;,
	\end{eqnarray}
	where the second equality follows from Lemma~\ref{prop:ComputingStuff}.
	
	The proof relies on the following key observation: if $\mu_a \geq \mu_b - \epsilon$,
	\begin{equation}\label{KeyToProve}
	\sup_{\alpha \in [0,1]} \inf_{\lambda \in (\mu^-,\mu^+-\epsilon)} \big[\alpha d(\mu_a,\lambda) + (1-\alpha)d(\mu_b,\lambda + \epsilon)\big] = d\big(\mu_a,\mu^*_\epsilon(\mu_a,\mu_b)\big)\;. 
	\end{equation}
	Using the  definition of $T^*_\epsilon(\bm\mu)$, Equality~\eqref{OptimRewriting} and the notation $b = 3 - a$ it holds that 
	\begin{eqnarray*}
		T^*_\epsilon(\bm\mu)^{-1} & = & \sup_{w \in \Sigma_2} \max_{a \in \cA_\epsilon} \inf_{\lambda \in (\mu^-,\mu^+-\epsilon)} \big[w_ad(\mu_a,\lambda) + w_bd(\mu_b,\lambda+\epsilon)\big] \\
		& = &   \max_{a \in \cA_\epsilon} \sup_{\alpha \in [0,1]} \inf_{\lambda \in (\mu^-\mu^+-\epsilon)} \big[\alpha d(\mu_a,\lambda) + (1-\alpha)d(\mu_b,\lambda+\epsilon)\big]\;.
	\end{eqnarray*}
	The result follows by using Equation~\eqref{KeyToProve} for all $a \in \cA_\epsilon$ (that is either equal to $\{1\},\{2\}$ or $\{1,2\}$).  
	
	We now prove \eqref{KeyToProve} as a double inequality. First, 
	\begin{align*}
	&\sup_{\alpha \in [0,1]} \inf_{\lambda \in (\mu^-,\mu^+-\epsilon)} \big[\alpha d(\mu_a,\lambda) + (1-\alpha)d(\mu_b,\lambda + \epsilon)\big] \\ &\leq \inf_{\lambda \in (\mu^-,\mu^+-\epsilon)} \max \big[d(\mu_a,\lambda),d(\mu_b,\lambda + \epsilon)\big] =  d(\mu_a,\mu^*_\epsilon(\mu_a,\mu_b))\;.
	\end{align*}
	To prove the second inequality, we denote by \[\lambda^*_{\alpha}(\mu_a,\mu_b) = \argmin{\lambda\in (\mu^-,\mu^+-\epsilon)} \  \big[\alpha d(\mu_a,\lambda) + (1-\alpha)d(\mu_b,\lambda + \epsilon)\big] \]
	and prove that there exists $\alpha \in [0,1]$ such that $\lambda^*_{\alpha}(\mu_a,\mu_b) = \mu^*_\epsilon(\mu_a,\mu_b)$. We first prove that the mapping $\alpha \mapsto \lambda^*_{\alpha}(\mu_a,\mu_b)$ is continuous. This can be established using that the minimizer of the convex function $\lambda \mapsto \alpha d(\mu_a,\lambda) + (1-\alpha)d(\mu_b,\lambda + \epsilon)$ which is attained on the compact set $[\mu_b - \epsilon,\mu_a]$, is unique together with the continuity of $(\alpha,\lambda) \mapsto \alpha d(\mu_a,\lambda) + (1-\alpha)d(\mu_b,\lambda + \epsilon)$. As $\lambda^*_0(\mu_a,\mu_b) = \mu_b - \epsilon$, $\lambda^*_1(\mu_a,\mu_b) = \mu_a$ and $\mu^*_\epsilon(\mu_a,\mu_b) \in (\mu_b - \epsilon,\mu_a)$, from the intermediate values theorem there exists $\alpha^*\in [0,1]$ such that $\lambda^*_{\alpha^*}(\mu_a,\mu_b) = \mu^*_\epsilon(\mu_a,\mu_b) \in (\mu_b - \epsilon,\mu_a)$.
	With this value $\alpha^*$ one can write\begin{align*}
	& \sup_{\alpha \in [0,1]} \inf_{\lambda \in (\mu^-,\mu^+-\epsilon)} \big[\alpha d(\mu_a,\lambda) + (1-\alpha)d(\mu_b,\lambda + \epsilon)\big] \\
	& \hspace{1cm}\geq \alpha^* d(\mu_a,\lambda_{\alpha^*}(\mu_a,\mu_b)) + (1-\alpha^*)d(\mu_b,\lambda_{\alpha^*}(\mu_a,\mu_b) + \epsilon) \\
	& \hspace{1cm} =  \alpha^* d(\mu_a,\mu^*_\epsilon(\mu_a,\mu_b)) + (1-\alpha^*)d(\mu_b,\mu^*_\epsilon(\mu_a,\mu_b) + \epsilon) \\
	& \hspace{1cm} =  d(\mu_a,\mu^*_\epsilon(\mu_a,\mu_b))\;,
	\end{align*}
	which together with the first inequality proves \eqref{KeyToProve}.

	\section{Proof of Theorem~\ref{thm:2armsGeneralSampling}}\label{proof:2armsGeneralSampling}
	
Assume to fix the ideas that $\mu_1 \geq \mu_2$.
		\paragraph{Case 1: ${\mu_1 \geq \mu_2+\epsilon}$} To simplify the notation we use the notation $\mu^*_\epsilon$ to denote the quantity $\mu^*_\epsilon(\mu_1,\mu_2)$ defined in Lemma~\ref{prop:AlternativeComplexity}.  Considering the alternative bandit model $\bm\lambda$ such that $\lambda_1 = \mu^{*}_\epsilon$ and $\lambda_2 = \mu^{*}_\epsilon +\epsilon$, one has $\lambda_1 - \lambda_2 = -\epsilon$, hence $\bP_{\bm\lambda}\left(\IHat = 1\right) \leq \delta$, while $\bP_{\bm\mu}\left(\IHat = 1\right) \geq 1 - \delta$. Inequality~\eqref{StrongConverse} in Lemma~\ref{lem:CD} yields 
		\begin{align*}
		\bE[N_1(\tau_\delta)]d(\mu_1,\lambda_1) + \bE[N_2(\tau_\delta)]d(\mu_2,\lambda_2) &\geq \mathrm{kl}(\delta,1-\delta) \\
		(\bE[N_1(\tau_\delta)] + \bE[N_2(\tau_\delta)])d(\mu_1,\mu^{*}_\epsilon) &\geq \mathrm{kl}(\delta,1-\delta) \\
		\bE[\tau_\delta] & \geq  \frac{1}{d(\mu_1,\mu^{*}_\epsilon)}\mathrm{kl}(\delta,1-\delta)\;,
		\end{align*}
		and the conclusion follows by letting $\delta$ go to zero and using $\mathrm{kl}(\delta,1-\delta) \sim \ln(1/\delta)$.
		
		\paragraph{Case 2: ${0<\mu_1 - \mu_2 < \epsilon}$} As in the proof of Theorem~\ref{thm:2armsGeneralSampling}, we fix $\beta \in (0,1)$ and prove that  
		\begin{equation}\lim_{\delta \rightarrow 0}\bP_{\bm \mu}\Big(\tau_\delta \leq \frac{1-\beta}{d_\epsilon^*(\mu_1,\mu_2)} \ln\left(\frac{1}{\delta}\right)\Big) = 0\;.\label{equ:HighProba}\end{equation}
		Introducing the notation $n_\delta = \lfloor \frac{1-\beta}{d_\epsilon^*(\mu_1,\mu_2)} \ln\left(\frac{1}{\delta}\right)\rfloor$ and $C_\delta = (\tau_\delta \leq n_\delta)$, we prove that $\lim_{\delta \rightarrow 0} \bP_{\bm\mu}(C_\delta)=0$.
		To prove this, we show that $\bP_{\bm\mu}\left(\cC_\delta,\IHat=1\right)$ and $\bP_{\bm\mu}\left(\cC_\delta,\IHat=2\right)$ tend to zero when $\delta$ goes to zero, by invoking Inequality~\ref{WeakConverse} in Lemma~\ref{lem:CD} with a different alternative model $\bm \lambda$ in each case. 
		
		To control $\bP_{\bm\mu}\left(\cC_\delta,\IHat=1\right)$ we consider the alternative model $\bm \lambda$ in which $\lambda_1 = \mu^{*}_\epsilon(\mu_1,\mu_2)$ and $\lambda_2 = \mu^{*}_\epsilon(\mu_1,\mu_2) + \epsilon$, under which $\bP_{\bm\lambda}(\IHat = 1) \leq \delta$. We let $L_t(\bm \mu,\bm \lambda)$ be the log-likelihood ratio of the observations under the two models: 
		\[L_t(\bm\mu,\bm\lambda) = \sum_{s=1}^{N_1(t)} \ln \frac{f_{\mu_1}(Y_{1,s})}{f_{\lambda_1}(Y_{1,s})} + \sum_{s=1}^{N_2(t)} \ln \frac{f_{\mu_2}(Y_{2,s})}{f_{\lambda_2}(Y_{2,s})}.\]
		The choice of $\bm \lambda$ permits to write
		\[\bE_{\bm\mu}\left[\ln \frac{f_{\mu_1}(Y_{1,s})}{f_{\lambda_1}(Y_{1,s}))}\right]=\bE_{\bm\mu}\left[\ln \frac{f_{\mu_2}(Y_{2,s})}{f_{\lambda_2}(Y_{2,s}))}\right] = d(\mu_1,\mu^{*}_\epsilon(\mu_1,\mu_2))\;,\]
		thus $M_t = L_t(\bm\mu,\bm\lambda) - t d( \mu_1,\mu^*_\epsilon(\mu_1,\mu_2))$ is a martingale whose increments have bounded variance, and that does not depend on $\delta$ (as the sampling rule $A_t$ is independent from $\delta$). By the law of large number for martingales, it holds that 
		\begin{equation}\frac{L_t(\bm\mu,\bm\lambda)}{t} \overset{\bP_{\bm\mu} - a.s.}{\underset{t\rightarrow \infty}{\longrightarrow}} d(\mu_1,\mu_\epsilon^{*}(\mu_1,\mu_2)).\label{equ:LLMart1}\end{equation}
		Proceeding as in the previous proofs, it follows from Inequality~\eqref{WeakConverse} with $x = (1-\beta/2) \ln (1/\delta)$ that
		\[\bP_\mu \left(C_\delta,\IHat = 1\right) \leq \delta^{\frac{\beta}{2}} + \bP_{\mu}\left(\frac{L_{n_\delta}(\mu,\epsilon)}{n_\delta} \geq \frac{1-\beta/2}{1-\beta} d_\epsilon^*(\mu_1,\mu_2)\right).\]
		The second term in the RHS tends to zero when $\delta$ goes to zero using \eqref{equ:LLMart1} and the fact that $n_\delta \rightarrow \infty$ and $\frac{1 - \beta/2}{1 - \beta} d^*_\epsilon(\mu_1,\mu_2) > d(\mu_1,\mu_\epsilon^{*}(\mu_1,\mu_2))$. Hence, one obtains $\lim_{\delta \rightarrow 0}\bP_\mu \left(C_\delta,\IHat = 1\right) = 0$.
		
		To control $\bP_{\bm\mu}\left(\cC_\delta,\IHat=2\right)$, we proceed similarly and consider the alternative model $\bm\lambda'$ in which $\lambda_1'=\mu_\epsilon^{*}(\mu_2,\mu_1)$ and $\lambda_2' = \mu_\epsilon^{*}(\mu_2,\mu_1) - \epsilon$, under which $\bP_{\bm \lambda'}(\IHat=2) \leq \delta$. For this particular choice, the log-likelihood ratio satisfies 
		\begin{equation}\lim_{t\rightarrow \infty}\frac{L_t(\bm\mu,\bm\lambda')}{t}\overset{\bP_{\bm\mu} - a.s.}{\underset{t\rightarrow\infty}{\longrightarrow}} d(\mu_2,\mu_\epsilon^{*}(\mu_2,\mu_1))\;.\label{eq:LLMart2}\end{equation}
		Applying Inequality~\ref{WeakConverse} as above yields 
		\[\bP_\mu \left(C_\delta,\IHat = 2\right) \leq \delta^{\frac{\beta}{2}} + \bP_{\mu}\left(\frac{L_{n_\delta}(\mu,\epsilon)}{n_\delta} \geq \frac{1-\beta/2}{1-\beta} d_\epsilon^*(\mu_1,\mu_2)\right)\]
		and noting that $\frac{1-\beta/2}{1-\beta}d_\epsilon^*(\mu_1,\mu_2) > d(\mu_2,\mu_\epsilon^{*}(\mu_2,\mu_1))$ one can prove using \eqref{eq:LLMart2} that the second term in the RHS goes to zero. Hence $\lim_{\delta \rightarrow 0}\bP_\mu \left(C_\delta,\IHat = 2\right) = 0$.

	\section{COMPUTATION OF THE OPTIMAL WEIGHTS}\label{proof:OptimalWeights}
	
	In this section, we prove Theorem~\ref{prop:OptimalWeights}. 
	Let $a\in \cA_\epsilon(\bm\mu)$ be an $\epsilon$-optimal such that $\mu_a > \mu_b - \epsilon$ and recall the definition of the interval 
	\[\cI_{a,b}^{\epsilon} = \big[\mu^-\vee (\mu_b-\epsilon),\mu_a \wedge(\mu^+-\epsilon)\big]\] and of the optimal weights
	\begin{align}
	w_\epsilon^{*,a}(\bm\mu) & =   \argmax{w \in \Sigma_K} \min_{b \neq a} \inf_{\lambda \in \cI_{a,b}^{\epsilon}} \big[w_ad(\mu_a,\lambda) + w_bd(\mu_b,\lambda+\epsilon)\big]. \nonumber\\
	& = \argmax{w \in \Sigma_K} \min_{b \neq a} w_a g_b\left(\frac{w_b}{w_a}\right)\;, \label{Optimumw}
	\end{align}
	where
	\begin{align*}g_b(x) & =  d\big(\mu_a, \lambda_b(x)) + x d(\mu_b, \lambda_b(x) + \epsilon\big) \ \ \text{and}\\
	\lambda_b(x) & =  \argmin{\lambda \in \cI_{a,b}^{\epsilon}} \big[d(\mu_a, \lambda) + x d(\mu_b, \lambda + \epsilon)\big]\;.
	\end{align*}
	
	\paragraph{A particular case} We first study the particular case in which $\mu_a \geq \mu^+-\epsilon$ and there exists $\tilde b \neq a$ such that $\mu_{\tilde b} = \mu^+$. Among common exponential families, this may happen for Bernoulli distribution with $\mu_a = \mu_{\tilde b} = 1$. 
	In that case $\cI_{a,{\tilde b}}^{\epsilon}$ is reduced to the singleton $\{\mu^+ - \epsilon\}$ and 
	\begin{align*}\lambda_{\tilde b}(x) &= d(\mu_a, \mu^+ - \epsilon) + x d(\mu^+,\mu^+) = d(\mu_a,\mu^+ - \epsilon)
	\end{align*}
	is a constant mapping.
	
	Observe that, for all $b\neq a$, as $\lambda \mapsto d(\mu_a, \lambda)$ is decreasing on $(\mu^-,\mu^+ - \epsilon)$,
	\begin{align*}
	w_a d(\mu_a,\mu^+-\epsilon) &\leq \inf_{\lambda \in (\mu^-,\mu^+ -\epsilon)} \big[w_a d(\mu_a,\lambda) + w_b d(\mu_b,\lambda + \epsilon)\big] \\
	&= \inf_{\lambda \in \cI_{a,b}^\epsilon} \big[w_a d(\mu_a,\lambda) + w_b d(\mu_b,\lambda + \epsilon)\big] \;. 
	\end{align*}
	Therefore, $\min_{b\neq a } w_ag_b(w_a/w_b) = w_ag_{\tilde b}(w_a/w_{\tilde b}) = w_a d(\mu_a,1-\epsilon)$ and 
	\[ w_\epsilon^{*,a}(\bm\mu) = \argmax{w \in \Sigma_K}w_a  d(\mu_a,\mu^+-\epsilon) = (\ind_{(b=a)})_{b=1,\dots,K} \;.\]
	This proves the first statement of Theorem~\ref{prop:OptimalWeights}. 
	
	\paragraph{General case} In the general case, $\cI_{a,b}^{\epsilon}$ is a non-degenerated compact interval, and one can define the following inverse mapping 
	\begin{eqnarray*}
		x_b : [d(\mu_a,\mu_a \wedge (\mu^+-\epsilon)),d(\mu_a,(\mu_b-\epsilon)\vee \mu^-)) & \longrightarrow & \R^+\\
		y & \mapsto & g_b^{-1}(y),
	\end{eqnarray*}
	which is increasing. 
	Let $w^*$ be an optimum in \eqref{Optimumw}, and define $x_b^*=w^*_b/w^*_a$ if $b\neq a$ and $x_a^* = 1$. Then 
	\[(x^*_b)_{b\neq a} \in \argmax{\bm x \in (\R^+)^{K-1}}\frac{\min_{b\neq a } g_b(x_b)}{1+\sum_{b\neq a} x_b}\]
	Since all the $g_b$ are increasing, such an optimum equalizes the values of $g_b(x^*_b)$: there exists $y$ such that $\forall b \neq a, \ \ g_b(x^*_b) = y$, and $y$ necessarily belongs to $\cI_{a}=[d(\mu_a,\mu_a \wedge (\mu^+-\epsilon)),d(\mu_a,(\max_{b\neq a}\mu_b-\epsilon)\vee \mu^-))$, the intersection of the image of all functions $g_b$. 
	
	\noindent By re-parameterizing the optimization problem by the common value $y$ of $g_b(x_b)$ (which yields $x_b = x_b(y)$ by definition of the inverse function $x_b$), one obtains $x_b^* = x_b(y^*)$ where 
	\[y^* \in \argmax{y \in \cI_{a}} \frac{y}{1+\sum_{b\neq a} x_b(y)}\;.\]
	
	We now justify that this maximum is unique, and provide a more explicit expression for it. The derivative of  
	$G:\cI_a\to \R$ defined by $G(y) = \frac{y}{1+ \sum_{b\neq a} x_b(y)}$ is 
	\begin{equation}G'(y) = \frac{1 + \sum_{b\neq a} x_b(y) - y \sum_{b\neq a} x_b'(y)}{\big(1+\sum_{b\neq a} x_b(y)\big)^2}\;.\label{calc:Derivative}\end{equation}
	In order to compute $x'_b(y)$ and simplify the expression of this derivative, we rely on the following properties of the minimizer $\lambda_b(x)$. 
	\begin{lemma}\label{prop:LambdaB}
		The mapping $x \mapsto \lambda_b(x)$ is differentiable on $\R^+$ and such that for all $x > 0$,  $\lambda'_b(x) < 0$. Moreover, 
		\[\lim_{x\rightarrow 0} \lambda_b(x) = \mu_a \wedge (\mu^+-\epsilon) \ \ \text{and} \ \ \lim_{x\rightarrow \infty} \lambda_b(x) = (\mu_b-\epsilon) \vee \mu^-\;.\]
	\end{lemma}
	\begin{proof}	
		For every $x$, the function $h_x(\lambda):=d(\mu_a,\lambda)+xd(\mu_b, \lambda+\epsilon)$ is strictly convex: its derivative $h'_x(\lambda)$ is increasing and crosses $0$ at $\lambda_b(x)$.
		The differentiability of $\lambda_b$  
		follows from that of $d$ by the implicit function theorem.  Since for every $\lambda>\mu_b-\epsilon$, $d'_2(\mu_b, \lambda+\epsilon)>0$, for every $y>x$
		\begin{align*}h'_y(\lambda_b(x)) &= d'_2\big(\mu_a, \lambda_b(x)\big)+yd'_2\big(\mu_b, \lambda_b(x)+\epsilon\big) \\&=  d'_2\big(\mu_a, \lambda_b(x)\big)+xd'_2\big(\mu_b, \lambda_b(x)+\epsilon\big) + (y-x)d'_2\big(\mu_b, \lambda_b(x)+\epsilon\big)\\&= 0 + (y-x)d'_2\big(\mu_b, \lambda_b(x)+\epsilon\big)>0\;. \end{align*} 
		Thus, for every $\alpha>0$ small enough,
		\[h'_y\big(\lambda_b(x)-\alpha\big) \geq (y-x)d'_2\big(\mu_b, \lambda_b(x)+\epsilon\big) -  \alpha \sup_{[\lambda_b(x)-\alpha, \lambda_b(x)]} h''_y >0\;. \]
		Since $\lambda_b(y)$ is the zero of the increasing function $h'_y$, this entails that
		\[ \lambda_b(y) \leq \lambda_b(x) - \frac{(y-x)d'_2(\mu_b, \lambda_b(x)+\epsilon)}{\sup_{[\lambda_b(y), \lambda_b(x)]}h''_y}\;,\] which shows that $\lambda'_b(x)<0$. 
		The limit when $x$ goes to $0$ follows from the continuity of $\lambda_b$, and when $x\to\infty$  it appears that 
		\begin{align*}\argmin{\lambda \in \cI_{a,b}^{\epsilon}} \big[d(\mu_a, \lambda) + x d(\mu_b, \lambda + \epsilon)\big] &= \argmin{\lambda \in \cI_{a,b}^{\epsilon}} \big[d(\mu_a, \lambda)/x + d(\mu_b, \lambda + \epsilon)\big] \\&\to \argmin{\lambda \in \cI_{a,b}^{\epsilon}} \big[ d(\mu_b, \lambda + \epsilon)\big]= (\mu_b-\epsilon) \vee \mu^-\;.\end{align*}
	\end{proof}
	
	As the minimum of the convex function $\lambda \mapsto d(\mu_a,\lambda)+d(\mu_b,\lambda + \epsilon)$ on $(\mu^-,\mu^+-\epsilon)$, $\lambda_b(x)$ cancels its derivative, hence 
	\[
	d_2'(\mu_a,\lambda_b(x)) + x d_2'(\mu_b,\lambda_b(x) + \epsilon)=0\;,
	\]
	where $d_2'(x,y) := \partial d(x,y) / \partial y$ denotes the partial derivative of $d(x,y)$ with respect to the second variable. Using this equality, the derivative of $g_b$ simplifies to 
	\begin{align*}
	g_b'(x) & =  d(\mu_b,\lambda_b(x) + \epsilon) + \lambda_b'(x)\big[d_2'(\mu_a,\lambda_b(x)) + xd_2'(\mu_b,\lambda_b(x)+\epsilon)\big] \\
	& =  d(\mu_b,\lambda_b(x) + \epsilon)\;.
	\end{align*}
	Hence for every $y \in \big[d(\mu_a,\mu_a \wedge (\mu^+-\epsilon)),d(\mu_a,(\mu_b-\epsilon)\vee \mu^-)\big)$, 
	\[x_b'(y) = \frac{1}{d\big(\mu_b,\lambda_b(x_b(y)) + \epsilon\big)}\;.\]
	Using this expression in the derivative~\eqref{calc:Derivative} yields 
	\begin{align*}G'(y) \geq 0 & \Leftrightarrow   1 + \sum_{b\neq a } x_b(y) \geq y \sum_{b \neq a } \frac{1}{d\big(\mu_b,\lambda_b(x_b(y))+\epsilon\big)}\\
	& \Leftrightarrow  \sum_{b\neq a}\frac{d\big(\mu_a,\lambda_b(x_b(y))\big)}{d\big(\mu_b,\lambda_b(x_b(y))+\epsilon\big)} \leq 1\;,\end{align*}
	where we use that by definition, for all $b$,
	\[y = g_b(x_b(y)) = d\big(\mu_a,\lambda_b(x_b(y))\big) + x_b(y) d\big(\mu_b,\lambda_b(x_b(y))+\epsilon\big)\;.\]
	
	The mapping \[F(y) := \sum_{b\neq a }\frac{d\big(\mu_a,\lambda_b(x_b(y))\big)}{d\big(\mu_b,\lambda_b(x_b(y))+\epsilon\big)}\] is increasing as a composition of the two decreasing mappings $y \mapsto \lambda_b\big(x_b(y)\big)$ (decreasing by Lemma~\ref{prop:LambdaB}) and  $\lambda \mapsto \sum_{b\neq a }\frac{d(\mu_a,\lambda)}{d(\mu_b,\lambda+\epsilon)}$. With $b^* \in \argmax{b\neq a} \mu_a$, Lemma~\ref{prop:LambdaB} permits to show that 
	\begin{align*}
	\lim_{y \rightarrow d(\mu_a,\mu_a\wedge(\mu^+ - \epsilon)))} \lambda_b(x_b(y)) &= \mu_a \wedge (\mu^+-\epsilon) \quad\text{ and }\\
	\lim_{y \rightarrow d(\mu_a,(\mu_{b^*}-\epsilon)\vee \mu^-)} \lambda_b(x_b(y)) &= (\mu_{b^*} - \epsilon) \vee \mu^-\;.
	\end{align*}
	Thus $F$ tends to $0$ in $d\big(\mu_a,\mu_a\wedge(\mu^+ - \epsilon)\big)$ and to $+\infty$ in $d\big(\mu_a,(\mu_{b^*}-\epsilon)\vee \mu^-\big)$. Hence, there is a unique $y^*$ that satisfies $F(y^*)=1$ and $G'$ is positive for $y < y^*$, negative for $y > y^*$ and satisfies $G'(y^*)=0$. 
	This proves that $G$ as a unique maximum on $\cI_a$, given by the unique solution $y^*$ to the equation 
	\[\sum_{b\neq a}\frac{d\big(\mu_a,\lambda_b(x_b(y))\big)}{d\big(\mu_b,\lambda_b(x_b(y))+\epsilon\big)} = 1\;,\]
	which justifies the second statement of Theorem~\ref{prop:OptimalWeights}.

\bibliography{biblioBandits}

\end{document}